\documentclass[a4paper,11pt]{amsart}
\usepackage[left=25mm, right=25mm, top=15mm, bottom=15mm, includeheadfoot]{geometry}

\usepackage[T1]{fontenc}

\usepackage[leqno]{amsmath}
\usepackage{amsthm}
\usepackage{amsfonts}
\usepackage{amssymb}

\usepackage{graphicx}
\usepackage[utf8]{inputenc}
\usepackage[english]{babel}

\theoremstyle{plain}
\newtheorem{lem}{Lemma}[section]
\newtheorem{thm}{Theorem}

\theoremstyle{remark}
\newtheorem{rem}[lem]{Remark}

\theoremstyle{definition}

\providecommand{\abs}[1]{\lvert#1\rvert} 
\providecommand{\norm}[1]{\lVert#1\rVert}

\DeclareMathOperator{\ud}{\!d\!}
\DeclareMathOperator{\Ud}{D}

\DeclareMathOperator{\supp}{supp}
\DeclareMathOperator{\esssup}{ess\, sup}
\DeclareMathOperator{\Div}{div}
\DeclareMathOperator{\Rot}{rot}
\DeclareMathOperator{\Dt}{\frac{\ud}{\ud t}}

\numberwithin{equation}{section}

\begin{document}

\title[Global existence of strong solutions]{Global existence of strong solutions to micropolar equations in cylindrical domains}
\author{Bernard Nowakowski}
\thanks{The author is partially supported by Polish KBN grant N N201 393137}
\address{Bernard Nowakowski\\ Institute of Mathematics\\ Polish Academy of Sciences\\ \'Snia\-deckich 8\\ 00-956 Warsaw\\ Poland}
\email{bernard@impan.pl}

\subjclass[2000]{35Q30, 76D05}

\keywords{micropolar fluids, cylindrical domains, global existence, strong solutions}
\begin{abstract}
	The micropolar equations are a useful generalization of the classical Navier-Stokes model for fluids with micro-structure. We prove the existence of global and strong solutions to these equations in cylindrical domains in $\mathbb{R}^3$. We do not impose any restrictions on the magnitude of the initial and external data but we require that they cannot change in the $x_3$-direction too fast.  
\end{abstract}

\maketitle

\section{Introduction}

Introduced in 1966 by A.~Eringnen (see \cite{erin}), micropolar equations became an important generalization of the classical Navier-Stokes model. These equations take into account that fluid molecules may rotate independently of the fluid rotation. Thus, the standard Navier-Stokes system is complemented with another vector equation which describes the angular momentum of the particles. If we denote the velocity field by $v$ and the microrotation fields by $\omega$, then we see that $(v,\omega)$ has six degrees of freedom. Let us clearly emphasize that $\omega$ does not represent the rotation field $(\Rot v)$ derived from the velocity field ($v$) and in most cases these vector fields differ fundamentally from each other. This phenomenon gains an immense significance for modelling some well-known fluids, e.g. animal blood or liquid crystals (see e.g. \cite{pop2}).

In the microscale, when at least one dimension of the domain is only a few times larger than the size of the molecules (e.g. blood vessels, lubricants), fluid motions even for isotropic fluid can differ substantially from what would follow from the computations based entirely on the Navier-Stokes equations (see \cite{shar}). This behavior is caused by the dominance of the surface stresses over body forces. Although not all the aspects of physical experiments have been fully explained but it is justified to assume that the surface stresses and the internal degrees of freedom of particles are the deciding factors for properties of fluid motion. 

It is worth mentioning, that apart from A. Eringen, other mathematicians and physicists have proposed numerous generalizations of the Navier-Stokes equations. The comparison of these theories can be found in \cite{ari}. For a short historical review we refer the reader to \cite[Ch. 1, \S 5]{luk1}. 

In this work we plan to investigate the global existence of strong solutions to micropolar equations which are given by
\begin{equation}
	\begin{aligned}\label{p1}
		&v_{,t} + v\cdot \nabla v - (\nu + \nu_r)\triangle v + \nabla p = 2\nu_r\Rot\omega + f   & &\text{in } \Omega^{\infty} := \Omega\times(0,\infty),\\
		&\Div v = 0 & &\text{in } \Omega^{\infty},\\
		&\omega_{,t} + v\cdot \nabla \omega - \alpha \triangle \omega - \beta\nabla\Div\omega + 4\nu_r\omega  = 2\nu_r\Rot v + g & &\text{in } \Omega^{\infty}, \\
		&v\vert_{t = 0} = v(0), \qquad \omega\vert_{t = 0} = \omega(0) & &\text{in $\Omega$}.	
    \end{aligned}
\end{equation}
The unknowns are: the velocity field $v$, the pressure $p$ and the microroation field $\omega$. The viscosity coefficients $\nu$, $\nu_r$, $\alpha$ and $\beta$ are fixed and positive. Note that if $\nu_r$ then \eqref{p1}$_{1,2}$ and \eqref{p1}$_3$ get uncoupled. Therefore we cannot expect better results than for the classical Navier-Stokes equations.

So far we have not specified the domain $\Omega$. We assume that is has a product form 
\begin{equation*}
	\left\{(x_1,x_2)\in \mathbb{R}^2\colon \varphi(x_1,x_2) \leq c_0\right\}\times\left\{x_3\colon -a \leq x_3 \leq a\right\},
\end{equation*}
where the constants $a$ and $c_0$ are positive and $\varphi$ is a $\mathcal{C}^2$-closed curve in $\mathbb{R}^2$. Thus,  $\Omega$ is a finite cylinder placed alongside the $x_3$-axis (see Figure \ref{fig1}). 
\begin{figure}[h!]
	\begin{center}
		\includegraphics[width=0.6\linewidth]{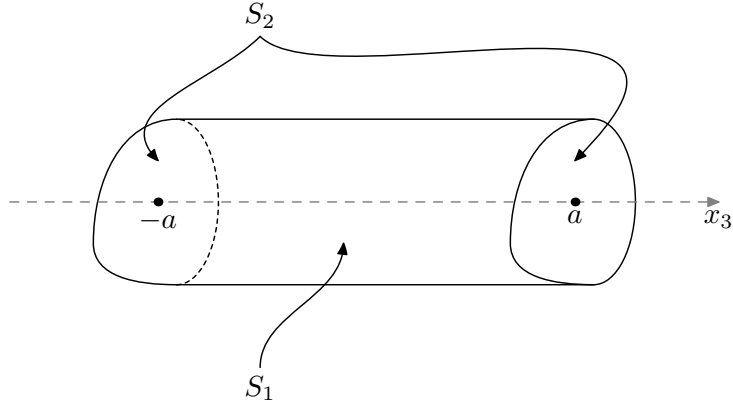}
		\caption{The domain $\Omega$ and its boundary $S: = S_1\cup S_2$.}\label{fig1}
	\end{center}
\end{figure}

From practical point of view (blood vessels, lubrication theory) our choice is justified. From theoretical perspective, our approach is intensely focused upon search for such solutions that are close to two dimensional (see e.g. \cite{wm2}, \cite{ren1}, \cite{wm6}). The solutions which are proved to exist, can be regarded as a slight perturbation of two dimensional flow along the perpendicular direction. This perturbation will be somehow measured by $\delta(t)$ (see \eqref{eq26}), which we introduce later. 

We shall emphasize that since we only require the initial rate of change of the flow and microrotation, as well as the derivatives of the external data with respect to $x_3$ to be small, the flow alongside the cylinder can be large, but close to constant. 

As far as the boundary condition are concerned, we use
\begin{equation}
		\begin{aligned}\label{p2}
			&v\cdot n = 0 & &\text{on $S^{\infty} := S\times(0,\infty)$}, \\
			&\Rot v \times n = 0  & &\text{on $S^{\infty}$}, \\
			&\omega = 0 & &\text{on $S_1^{\infty}$}, \\
			&\omega' = 0, \qquad \omega_{3,x_3} = 0 & &\text{on $S_2^{\infty}$},
    \end{aligned}    
\end{equation}
where $n$ is the unit outward vector. The first two equations may be interpreted as tangential ``slip'' velocity being proportional to tangential stress with a factor of proportionality depending only on the curvature of $\varphi$ (see e.g. \cite{Clopeau:1998vj}, \cite{kell}). Such boundary condition was already postulated in 1827 by C.M.L.H. Navier. The third equation is clear but from the physical point view not necessarily adequate because the molecules may not move but they can rotate (see \cite{bou}). This  effect is regarded in the fourth equation (see \cite{mig}). 

\section{Notation}

Before we present the main result of this work, we should employ some notation. By $\Omega^t$ we denote $\Omega\times(t_0,t)$ where $0 \leq t_0 < t < \infty$. A generic constant $c$ may change from line to line and is subscripted with appropriate symbol which indicates the dependence on the domain ($c_{\Omega}$), embedding theorems ($c_I$), the Poincar\'e inequality ($c_P$) and viscosity coefficients ($c_{\alpha,\beta,\nu,\nu_r}$). To simplify the formulas we will also use
\begin{align*}
	&h := v_{,x_3}, & & &\theta := \omega_{,x_3}.
\end{align*}

To shorten energy estimate we introduce
\begin{equation}\label{eq14}
	\begin{aligned}
		E_{v,\omega}(t) &:= \norm{f}_{L_2(t_0,t;L_{\frac{6}{5}}(\Omega))} + \norm{g}_{L_2(t_0,t;L_{\frac{6}{5}}(\Omega))} + \norm{v(t_0)}_{L_2(\Omega)} + \norm{\omega(t_0)}_{L_2(\Omega)}, \\
		E_{h,\theta}(t) &:= \norm{f_{,x_3}}_{L_2(t_0,t;L_{\frac{6}{5}}(\Omega))} + \norm{g_{,x_3}}_{L_2(t_0,t;L_{\frac{6}{5}}(\Omega))} + \norm{h(t_0)}_{L_2(\Omega)} + \norm{\theta(t_0)}_{L_2(\Omega)}.
	\end{aligned}
\end{equation}
Finally, the function
\begin{equation}\label{eq26}
	\delta(t) := \norm{f_{,x_3}}^2_{L_2(\Omega^t)} + \norm{g_{,x_3}}^2_{L_2(\Omega^t)} + \norm{\Rot h(t_0)}^2_{L_2(\Omega)} + \norm{h(t_0)}^2_{L_2(\Omega)} + \norm{\theta(t_0)}^2_{L_2(\Omega)}
\end{equation}
will be of particular interest. It expresses the smallness assumption which has to be made in order to prove the existence of regular solutions on $(t_0,t)$. Note that it contains only derivative of the external and the initial data with respect to $x_3$.

The notation for function spaces is standard and follows \cite[Ch. 3, \S 1.1]{luk1}, \cite[Ch. 2, \S 3]{lad} and \cite[Ch. 1, \S 1.1]{tem}: 

\begin{enumerate}
	\item[$\bullet$] $W^m_p(\Omega)$, where $m \in \mathbb{N}$, $p \geq 1$, is the closure of $\mathcal{C}^{\infty}(\Omega)$ in the norm
		\begin{equation*}
			\norm{u}_{W^m_p(\Omega)} = \left(\sum_{\abs{\alpha} \leq m} \norm{\Ud^{\alpha} u}_{L_p(\Omega)}^p\right)^{\frac{1}{p}},
		\end{equation*}
	\item[$\bullet$] $H^k(\Omega)$, where $k \in \mathbb{N}$, is simply $W^k_2(\Omega)$, 
	\item[$\bullet$] $W^{2,1}_p(\Omega^t)$, where $p \geq 1$, is the closure of $\mathcal{C}^{\infty}(\Omega\times(t_0,t_1))$ in the norm
		\begin{equation*}
			\norm{u}_{W^{2,1}_p(\Omega^t)} = \left(\int_{t_0}^{t_1}\!\!\!\int_{\Omega} \abs{u_{,xx}(x,s)}^p + \abs{u_{,x}(x,s)}^p + \abs{u(x,s)}^p + \abs{u_{,t}(x,s)}^p\, \ud x\, \ud s\right)^{\frac{1}{p}},
		\end{equation*}
	\item[$\bullet$] $H^1_0(\Omega)$ is the closure of $\mathcal{C}^{\infty}_0(\Omega)$ in the norm
		\begin{equation*}
			\norm{u}_{H^1_0(\Omega)} = \left(\int_{\Omega} \abs{\nabla u(x)}^2\, \ud x\right)^{\frac{1}{2}},
		\end{equation*}
	\item[$\bullet$] $L_q(t_0,t_1;X)$, where $q \geq 1$ and $X$ is a Banach space, is the set of all strongly measurable functions defined on the interval $[t_0,t_1]$ with values in $X$ with finite norm defined by
		\begin{equation*}
			\norm{u}_{L_q(t_0,t_1;X)} = \left(\int_{t_0}^{t_1}\norm{u(s)}_X^q\, \ud s\right)^{\frac{1}{q}},
		\end{equation*}
		where $1 \leq p < \infty$ and by
		\begin{equation*}
			\norm{u}_{L_{\infty}(t_0,t_1;X)} = \underset{t_0\leq s \leq t_1}{\esssup}\norm{u(s)}_X,
		\end{equation*}
		for $q = \infty$.
	\item[$\bullet$] $V^k_2(\Omega^t)$, where $k \in \mathbb{N}$, is the closure of $\mathcal{C}^{\infty}(\Omega\times(t_0,t_1))$ in the norm
		\begin{equation*}
			\norm{u}_{V^k_2(\Omega^t)} = \underset{t\in (t_0,t_1)}{\esssup}\norm{u}_{H^k(\Omega)} \\
    +\left(\int_{t_0}^{t_1}\norm{\nabla u}^2_{H^{k}(\Omega)}\, \ud t\right)^{1/2}
		\end{equation*}
\end{enumerate}

\section{Main result}

The main results of this work reads:

\begin{thm}[global existence]\label{t2}
	Let $t_0 = 0$ and $0 < T < \infty$ be sufficiently large and fixed. Suppose that $v(0), \omega(0) \in H^1(\Omega)$ and $\Rot h(0) \in L_2(\Omega)$. In addition, let the external data satisfy $f_3\vert_{S_2} = 0$, $g'\vert_{S_2} = 0$, 
	\begin{align*}
		&\norm{f(t)}_{L_2(\Omega)} \leq \norm{f(kT)}_{L_2(\Omega)}e^{-(t - kT)}, & & &\norm{f_{,x_3}(t)}_{L_2(\Omega)} &\leq \norm{f_{,x_3}(kT)}_{L_2(\Omega)}e^{-(t - kT)},\\
		&\norm{g(t)}_{L_2(\Omega)} \leq \norm{g(kT)}_{L_2(\Omega)}e^{-(t - kT)}, & & &\norm{g_{,x_3}(t)}_{L_2(\Omega)} &\leq \norm{g_{,x_3}(kT)}_{L_2(\Omega)}e^{-(t - kT)}
	\end{align*}
	and 
	\begin{equation*}
		\sup_k \left\{f(kT),f_{,x_3}(kT),g(kT),g_{,x_3}(kT)\right\} < \infty.
	\end{equation*}
	Then, for $\delta(T)$ sufficiently small there exists a unique and regular solution to problem \eqref{p1} equipped with the boundary conditions \eqref{p2} on the interval $(0,\infty)$. Moreover, 
	\begin{multline*}
		\norm{v}_{W^{2,1}_2(\Omega^{\infty})} + \norm{\omega}_{W^{2,1}_2(\Omega^{\infty})} + \norm{\nabla p}_{L_2(\Omega^{\infty})} \leq \sup_k \Big(\norm{f}_{L_2(\Omega^{kT})} + \norm{f_{,x_3}}_{L_2(\Omega^{kT})} \\
		+ \norm{g}_{L_2(\Omega^{kT})} + \norm{g_{,x_3}}_{L_2(\Omega^{kT})} + \norm{v(0)}_{H^1(\Omega)} + \norm{\omega(0)}_{H^1(\Omega)} + 1\Big)^3.
	\end{multline*}
\end{thm}

The proof of this theorem is based on the result obtained in \cite[Theorem 1]{2012arXiv1205.4046N}:
\begin{thm}[large time existence]\label{t1}
	Let $E_{v,\omega}(t) < \infty$, $E_{h,\theta}(t) < \infty$. Suppose that $v(t_0), \omega(t_0) \in H^1(\Omega)$, $f, g \in L_2(\Omega^t)$. Finally, assume that $f_3\vert_{S_2} = 0$, $g'\vert_{S_2} = 0$. Then, for $\delta(t)$ sufficiently small there exists a unique solution $(v,\omega) \in W^{2,1}_2(\Omega^t)\times W^{2,1}_2(\Omega^t)$ to problem \eqref{p1} supplemented with the boundary conditions \eqref{p2} such that
	\begin{equation*}
			\norm{v}_{W^{2,1}_2(\Omega^t)} + \norm{\nabla p}_{L_2(\Omega^t)} \leq c_{\alpha,\nu,\nu_r,I,P,\Omega}\Big(E_{v,\omega}(t) + E_{h,\theta}(t) + \norm{f'}_{L_2(\Omega^t)} + \norm{v(t_0)}_{H^1(\Omega)} + 1\Big)^3
	\end{equation*}
	and
	\begin{multline*}
			\norm{\omega}_{W^{2,1}_2(\Omega^t)} \leq c_{\alpha,\nu,\nu_r,I,P,\Omega}\Big(E_{v,\omega}(t) + E_{h,\theta}(t) + \norm{f'}_{L_2(\Omega^t)} + \norm{g}_{L_2(\Omega^t)}\\
			+ \norm{v(t_0)}_{H^1(\Omega)} + \norm{\omega(t_0)}_{H^1(\Omega)} + 1\Big)^3.
	\end{multline*}
\end{thm}

We see that in view of Theorem \ref{t1} the extension of solutions up to the infinity with respect to time can be done in many ways. A substantive argument would be: if the solution to \eqref{p1} on $[0,T]$ is regular and its estimate does not contain any constant which depend on time, then it is global. We must emphasize that the existing terminology on the topic is not precise. According to some authors a solution is global if the constants are time-dependent, but they do not blow up for any finite $T$. However, in such case it is more accurate to speak about large time existence instead of global existence.

For \eqref{p1} we could not simply put $T = \infty$, because it would lead to improper integrals and several technical difficulties. Besides, it would imply that the external data must vanish as $T$ goes to infinity. Hence, we adopt an alternative approach. We consider local solution on the time interval of the form $[kT,(k + 1)T]$, where $k \in \mathbb{N}$ and $T > 0$ is fixed number. Starting with $k = 0$ we let $k \to \infty$, thereby obtaining a sequence of solutions with different initial conditions $v(kT)$ and $\omega(kT)$. In order to guarantee that this sequence is in fact an extension of solution from $[0,T]$, we must control the growth of the initial conditions. Additionally, if $f \neq 0$ and $g \neq 0$, certain restrictions on the external data must be also imposed.

Let us now briefly discuss previous results concerning the existence of global solutions of strong solutions to micropolar equations. In \cite{lang1} and \cite{lang2} Lange proved the existence and uniqueness of global and strong solutions to \eqref{p1} in Hilbert spaces when the data are small enough. His approach is based on integral equations. This problem was also studied by Sava in \cite{sav2} under homogeneous boundary conditions in the case when the body forces and moments are not present. When the external data is present, $\nu$ is large and the data are small in comparison to $\nu$ problem \eqref{p1} with zero Dirichlet boundary condition was examined by Łukaszewicz in \cite{luk11} and in slightly more general framework by Ortega-Torres and Rojas-Medar in \cite{ort}, who in contrast to Łukaszewicz did not assume any decay for the external data as $t$ goes to the infinity. Semi-group approach was explored by Yamaguchi in \cite{yam}. He established the global existence of strong solutions in case of small data.

For further bibliographical notes we refer the reader to \cite{roj}, \cite{ort} and \cite[Ch.3, \S 5]{luk1}.

Summing up: our result is established under smallness assumption not on the data itself but on its rate of change along $x_3$-variable. The external data do not vanish as $t$ tends to infinity but exponentially decay on time intervals of the form $[kT,(k + 1)T]$. The boundary conditions for the velocity field and partially for the microrotation field belong to the class of slip boundary conditions. 

The rest of this work consists of Section \ref{sec4}, which contains some technical remarks, and Section \ref{sec5} where we give estimates needed to prove Theorem \ref{t2}.

\section{Auxiliary tools}\label{sec4}

In this Section we present some technical facts which are essential for further considerations.

\begin{lem}[On integration by parts]\label{l1}
	Let $u$ and $w$ belong to $H^1(\Omega)$. Then
	\begin{align*}
		\int_{\Omega} \Rot u \cdot w\, \ud x &= \int_{\Omega} \Rot w \cdot u\, \ud x + \int_S u \times n \cdot w\, \ud S \\
		&= \int_{\Omega} \Rot w \cdot u\, \ud x - \int_S w \times n \cdot u\, \ud S
	\end{align*}
\end{lem}
\begin{proof}
	It is an easy exercise.
\end{proof}

\begin{lem}\label{l2}
	Suppose that
		\begin{equation*}
			\begin{aligned}
				&\Rot u = \alpha & &\text{in $\Omega$}, \\
				&\Div u = \beta & &\text{in $\Omega$}
		\end{aligned}
	\end{equation*}
	with either $u \cdot n\vert_S = 0$ or $u \times n\vert_S = 0$.
	Then
	\begin{equation*}
		\norm{u}_{H^{k + 1}(\Omega)} \leq c_{\Omega}\big(\norm{\alpha}_{H^k(\Omega)} + \norm{\beta}_{H^k(\Omega)}).
	\end{equation*}
\end{lem}
\begin{proof}
	For the proof we refer the reader \cite{sol1}, where general overdetermined elliptic systems are examined. In particular, the case of tangent components of $u$ is considered. 
\end{proof}

\begin{lem}\label{lem12}
	Let $v$, $\theta$ and $f_{,x_3}$ be given. Then the pair $(h,q)$ is a solution to the problem
	\begin{equation}\label{p5}
		\begin{aligned}
			&h_{,t} - (\nu + \nu_r)\triangle h + \nabla q = -v \cdot \nabla h - h \cdot \nabla v + 2\nu_r\Rot \theta + f_{,x_3} & &\text{in $\Omega^t$}, \\
			&\Div h = 0 & &\text{in $\Omega^t$}, \\
			&\Rot h \times n = 0, \quad h\cdot n = 0 & &\textrm{on $S_1^t$}, \\
			&h' = 0, \qquad h_{3,x_3} = 0 & &\textrm{on $S_2^t$}, \\
			&h\vert_{t = t_0} = h(t_0) & &\textrm{in $\Omega$}.
		\end{aligned}
	\end{equation}
\end{lem}

\begin{lem}\label{lem4}
    Let $h$, $\omega$, $g'$ and $g_{,x_3}$ be given. Then the function $\theta$ is solution to the problem
    \begin{equation}\label{p6}
        \begin{aligned}
            &\theta_{,t} - \alpha \triangle \theta - \beta \nabla \Div \theta + 4\nu_r \theta = -h\cdot \nabla \omega - v \cdot \nabla \theta + 2 \nu_r \Rot h + g_{,x_3} & &\text{in $\Omega^t$}, \\
            &\theta = 0 & &\text{on $S_1^t$}, \\
            &\theta_3 = 0, \qquad \theta'_{,x_3} = -\frac{1}{\alpha}g' & &\text{on $S_2^t$}, \\
            &\theta\vert_{t = t_0} = \theta(t_0) & &\text{in $\Omega$}.
        \end{aligned}
    \end{equation}
\end{lem}
Proofs of Lemmas \ref{lem12} and \ref{lem4} are given in \cite{2012arXiv1205.4046N}. 

\begin{rem}
	Let us notice that for the functions $h$ and $\theta$ the Poincar\'e inequality holds. Indeed, since $h'$ vanishes on $S_2$ we only need to check if the integral of $h_3$ over $\Omega$ equals zero. We have
	\begin{equation*}
		\int_{\Omega} h_3\, \ud x = \int_{\Omega} v_{3,x_3}\, \ud x = \int_{S_2(x_3 = -a)} v_3\, \ud x' - \int_{S_2(x_3 = a)} v_3\, \ud x' = 0.
	\end{equation*}
	For $\theta_3$, which vanishes on $S_2$ it is also obvious. For $\theta'$ we simply calculate the mean value:
	\begin{equation*}
		\int_{\Omega} \theta'\, \ud x = \int_{\Omega} \omega'_{,x_3}\, \ud x = \int_{S_2(x_3 = -a)}\omega'\, \ud x' - \int_{S_2(x_3 = a)}\omega'\, \ud x' = 0,
	\end{equation*}
	which follows from \eqref{p2}$_4$.
\end{rem}

\begin{lem}\label{lem110}
	Suppose that $h$ is a solution to \eqref{p5}. Then
	\begin{equation*}
		\norm{h}_{H^2(\Omega)} \leq c_{\Omega}\norm{\triangle h}_{L_2(\Omega)}.
	\end{equation*}
\end{lem}
\begin{proof}
	Consider 
	\begin{equation*}
		\begin{aligned}
			&\Rot \Rot h = \alpha & &\text{in $\Omega$}, \\
			&\Div h = 0 & &\text{in $\Omega$}, \\
			&\Rot h \times n = 0 & &\text{on $S_1$},\\
			&\Rot h \cdot n = 0 & &\text{on $S_2$}.
		\end{aligned}
	\end{equation*}
	Introduce a partition of unity $\sum_{k = 1}^N \zeta_k(x_3) = 1$. If we denote $\bar u = \Rot h \zeta_k$, then the above system becomes
	\begin{equation*}
		\begin{aligned}
			&\Rot \bar u = \bar \alpha + 2\nabla h\cdot \nabla \zeta_k + h \triangle \zeta_k & &\text{in $\supp\zeta_k\cap\Omega$}, \\
			&\Div \bar u = 0 & &\text{in $\supp\zeta_k\cap\Omega$}, \\
			&\bar u \times n = [0,0,h_1\zeta_{k,x_3}n_1 + h_2\zeta_{k,x_3}n_2] & &\text{on $\supp\zeta_k\cap S_1$}, \\
			&\bar u \cdot n = 0 & &\text{on $\supp\zeta_k\cap S_2$}.
		\end{aligned}
	\end{equation*}
	Note, that the boundary condition on $\supp\zeta_k\cap S_1$ is equal to zero which follows from \eqref{p5}$_3$.

	There are four cases to consider:
	\begin{description}
		\item[1. $\supp \zeta_k \cap S = \emptyset$.] The boundary conditions is $\bar u = 0$. From Lemma \ref{l2} we get
			\begin{equation}\label{eq115}
				\norm{\bar u}_{H^1(\supp\zeta_k\cap\Omega)} \leq c_{\Omega} \left(\norm{\bar \alpha}_{L_2(\supp\zeta_k\cap\Omega)} + \norm{\nabla h}_{L_2(\supp\zeta_k\cap\Omega)} + \norm{h}_{L_2(\supp\zeta_k\cap\Omega)}\right).
			\end{equation}
		\item[2. $\supp \zeta_k \cap S_1 \neq \emptyset$, $\supp \zeta_k \cap S_2 = \emptyset$.] On the set $\supp \zeta_k \cap S_1$ we see that $\bar u \times n = 0$, whereas on $\partial(\supp\zeta_k\cap\Omega)\setminus(\supp\zeta_k\cap S_1)$ we have $\bar u = 0$ which in particular means that $\bar u \times n = 0$. Next we transform the set $\supp \zeta_k\cap \Omega$ into the half-space and utilize Theorem 5.5 from \cite{sol2}. It yields \eqref{eq115} but in the half-space, i.e. 
			\begin{equation*}
				\norm{\bar u}_{H^1\left(\mathbb{R}^{n - 1}_+\right)} \leq c_{\Omega} \left(\norm{\bar \alpha}_{L_2\left(\mathbb{R}^{n - 1}_+\right)} + \norm{\nabla h}_{L_2\left(\mathbb{R}^{n - 1}_+\right)} + \norm{h}_{L_2\left(\mathbb{R}^{n - 1}_+\right)}\right).
			\end{equation*}

		\item[3. $\supp \zeta_k \cap S_1 \neq \emptyset$, $\supp \zeta_k \cap S_2 \neq \emptyset$.] Observe that $\bar u \cdot n\vert_{S_2} = \Rot \bar h \cdot n\vert_{S_2} = 0 \Leftrightarrow h_{2,x_1} - h_{1,x_2} = 0$. On the other hand we already know that $h'\vert_{S_2} = 0$ and $h_{3,x_3} = 0$ on $S_2$ (see \eqref{p5}$_4$). Therefore we can reflect the function $h$ outside the cylinder following the formula
			\begin{equation*}
				\check h (x) = \begin{cases} \bar h(x) & x_3 \in \overline{\supp\zeta_k\cap \Omega}, \\
													(\bar h'(\bar x),-\bar h_3(\bar x)) & x_3 \leq -a, \\
													(\bar h'(\tilde x), -\bar h_3(\tilde x)) & x_3 \geq a,\end{cases}
			\end{equation*}
			where $\bar x = (x',-2a - x_3)$ and $\tilde x = (x', 2a - x_3)$. We easily check that $\check u \times n = 0$ on $\supp\zeta_k\cap \check S_1$ and $\check u = 0$ on $\partial(\supp\zeta_k\cap \check \Omega)$. Then we follow Case $2$.
		\item[4. $\supp \zeta_k \cap S_1 = \emptyset$, $\supp \zeta_k \cap S_2 \neq \emptyset$.] On $\supp\zeta_k\cap S_2$ we have $\bar u \cdot n = 0$ and on $\partial(\supp\zeta_k\cap\Omega)$ we get $\bar u = 0$. Next we map $\supp \zeta_k\cap \Omega$ into the half-space and utilize Theorem 5.5 from \cite{sol2}. It yields \eqref{eq115} in the half-space.
	\end{description}
   Summing over $k$ yields 
    \begin{equation*}
        \norm{\Rot h}_{H^1(\Omega)} \leq c_{\Omega}\left(\norm{\alpha}_{L_2(\Omega)} + \norm{h}_{H^1(\Omega)}\right) = c_{\Omega} \left(\norm{\triangle h}_{L_2(\Omega)} + \norm{h}_{H^1(\Omega)}\right).
    \end{equation*}
    From the above inequality and using (see \cite[Rem. 8.3]{2012arXiv1205.4046N})
	\begin{equation}\label{eq120}
		\norm{h}_{H^{k + 1}(\Omega)} \leq c_{\Omega}\norm{\Rot h}_{H^k(\Omega)}
	\end{equation}
	we deduce that for $\alpha = \Rot h \in H^1(\Omega)$ 
	\begin{equation*}
		\norm{h}_{H^2(\Omega)} \leq c_{\Omega} \norm{\alpha}_{H^1(\Omega)} = c_{\Omega} \norm{\Rot h}_{H^1(\Omega)} \leq c_{\Omega}\left(\norm{\triangle h}_{L_2(\Omega)} + \norm{h}_{H^1(\Omega)}\right).
	\end{equation*}
	Eventually, we demonstrate that $\norm{h}_{H^1(\Omega)} \leq \norm{\triangle h}_{L_2(\Omega)}$. Again, from \eqref{eq120} it follows that
	\begin{equation*}
		\norm{h}^2_{H^1(\Omega)} \leq c_{\Omega}\norm{\Rot h}^2_{L_2(\Omega)} = c_{\Omega}\int_{\Omega} \Rot h\cdot \Rot h\, \ud x = \int_{\Omega} \Rot\Rot h\cdot h\, \ud x - \int_S \Rot h \times n \cdot h\, \ud S,
	\end{equation*}
	where we also used Lemma \ref{l1}. The boundary integral vanishes on $S_1$ due to boundary conditions \eqref{p5}$_3$. On $S_2$ it can be written in a form	
	\begin{equation*}
		-\int_{S_2} h\times n \cdot \Rot h\, \ud x
	\end{equation*}
	and since $h\times n\vert_{S_2} = [-h_2,h_1,0]\vert_{S_2} = 0$, where the last equality follows from \eqref{p5}$_4$, we get
	\begin{equation*}
		\norm{h}^2_{H^1(\Omega)} \leq c_{\Omega}\int_{\Omega}\Rot\Rot h\cdot h\, \ud x = - c_{\Omega}\int_{\Omega} \triangle h\cdot h\, \ud x.
	\end{equation*}
	Now we use the H\"older and the Young with $\epsilon$ inequalities, which results in
	\begin{equation*}
		\norm{h}^2_{H^1(\Omega)} \leq c_{\Omega} \norm{\triangle h}_{L_2(\Omega)}^2
	\end{equation*}
	and ends the proof.
\end{proof}

\begin{lem}\label{l4}
	Let $E_{v,\omega}(t) < \infty$  (see \eqref{eq14}$_1$). Then for any $t_0 \leq t \leq t_1$ we have
	\begin{equation*}
		\norm{v}_{V_2^0(\Omega^t)} + \norm{\omega}_{V_2^0(\Omega^t)} \leq c_{\alpha,\nu,I,\Omega} E_{v,\omega}(t).
    \end{equation*}
\end{lem}

\begin{proof}
	This was proved in \cite[Lemma 8.1]{2012arXiv1205.4046N}.
\end{proof}

\section{Uniform estimates of solutions}\label{sec5}

We begin with certain refinement of the fundamental energy estimate for the function $v$ and $\omega$ in the norm $L_{\infty}(t_0,t_1;L_2(\Omega))$. 

\begin{lem}\label{lem28}
	Suppose that $v(t_0), \omega(t_0) \in L_2(\Omega)$ and $f,g \in L_{\infty}(t_0,t;L_2(\Omega))$. Then
	\begin{equation*}\tag{$\mathbf{A}$}
		\begin{aligned}
			\norm{v(t)}_{L_2(\Omega)}^2 + \norm{\omega(t)}_{L_2(\Omega)}^2 \leq c_{\alpha,\nu,\Omega}\left(\norm{f}_{L_{\infty}(t_0,t;L_2(\Omega))}^2 + \norm{g}_{L_{\infty}(t_0,t;L_2(\Omega))}^2\right)& \\ 
			+ \left(\norm{v(t_0)}^2_{L_2(\Omega)} + \norm{\omega(t_0)}^2_{L_2(\Omega)}\right)e^{-\frac{\min\{\alpha,\nu\}}{c_{\Omega}}(t - t_0)}.&
		\end{aligned}
	\end{equation*}
	If, in addition we assume that
	\begin{align*}
		\norm{f(t)}_{L_2(\Omega)} &\leq \norm{f(t_0)}_{L_2(\Omega)}e^{-(t - t_0)}, \\
		\norm{g(t)}_{L_2(\Omega)} &\leq \norm{g(t_0)}_{L_2(\Omega)}e^{-(t - t_0)}
	\end{align*}
	then $(\mathbf{A})$ implies 
	\begin{equation*}\tag{$\mathbf{B}$}
		\begin{aligned}
			\norm{v(t)}^2_{L_2(\Omega)} + \norm{\omega(t)}^2_{L_2(\Omega)} \leq \frac{c_{\Omega}}{\nu} \norm{f(t_0)}^2_{L_2(\Omega)}e^{-(t - t_0)} + \frac{c_{\Omega}}{\alpha} \norm{g(t_0)}^2_{L_2(\Omega)}e^{-(t - t_0)}& \\
		+ \left(\norm{v(t_0)}^2_{L_2(\Omega)} + \norm{\omega(t_0)}^2_{L_2(\Omega)}\right)e^{-{\frac{\min\{\alpha,\nu\}}{c_{\Omega}}}(t - t_0)}.&
		\end{aligned}	
	\end{equation*}
\end{lem}
To prove the above lemma we essentially follow the standard way of obtaining the basic energy estimates, however the final estimate is calculated in a slightly different way. 
\begin{proof}
	Multiplying \eqref{p1}$_{1,2}$ by $v$ and $\omega$ respectively and integrating over $\Omega$ yields
	\begin{multline*}
		\frac{1}{2}\Dt \left(\norm{v}^2_{L_2(\Omega)} + \norm{\omega}^2_{L_2(\Omega)}\right) - (\nu + \nu_r)\int_{\Omega} \triangle v\cdot v\, \ud x - \alpha\int_{\Omega} \triangle \omega \cdot \omega\, \ud x - \beta\int_{\Omega}\nabla \Div \omega\cdot \omega\, \ud x \\
		+ \int_{\Omega} \nabla p \cdot v\, \ud x + \int_{\Omega} v\cdot \nabla v \cdot v\, \ud x + \int_{\Omega} v\cdot \nabla \omega\cdot \omega\, \ud x + 4\nu_r\norm{\omega}^2_{L_2(\Omega)} \\
		= 2\nu_r\int_{\Omega}\Rot \omega\cdot v\, \ud x + 2\nu_r\int_{\Omega} \Rot v \cdot \omega\, \ud x + \int_{\Omega} f\cdot v\, \ud x + \int_{\Omega} g\cdot \omega\, \ud x.
	\end{multline*}
	In view of \eqref{p2}$_1$ and from $\Div v = 0$ we immediately get that
	\begin{align*}
		\int_{\Omega} v\cdot \nabla v \cdot v\, \ud x &= 0, \\
		\int_{\Omega} v\cdot \nabla \omega\cdot \omega\, \ud x &= 0, \\
		\int_{\Omega} \nabla p \cdot v\, \ud x &= 0.
	\end{align*}
	By the application of Lemma \ref{l1} we see that
	\begin{equation*}
		-\int_{\Omega} \triangle v \cdot v\, \ud x = \int_{\Omega} \Rot\Rot v\cdot v\, \ud x = \int_{\Omega} \abs{\Rot v}^2 + \int_S \Rot v\times n \cdot v\, \ud S
	\end{equation*}
	and
	\begin{multline*}
		-\int_{\Omega} \triangle \omega\cdot \omega, \ud x = \int_{\Omega} \Rot\Rot \omega\cdot \omega - \nabla \Div\omega \cdot \omega\, \ud x = \int_{\Omega} \abs{\Rot \omega}^2 + \abs{\Div \omega}^2\, \ud x \\
		+ \int_S \Rot\omega\times n\cdot\omega\,\ud S + \int_S \Div \omega (\omega\cdot n)\, \ud S.
	\end{multline*}
	In view of the boundary conditions \eqref{p2} all above boundary integrals vanish. On the right-hand side we have
	\begin{equation*}
		\int_{\Omega} \Rot \omega \cdot v\, \ud x = \int_{\Omega} \Rot v \cdot \omega\, \ud x + \int_S \omega\times n \cdot v\, \ud S,
	\end{equation*}
	where we applied Lemma \ref{l1}. Next we see that the boundary integral vanishes on $S_1$ since $\omega\vert_{S_1} = 0$ and on $S_2$ $\omega\times n = [-\omega_2,\omega_1,0] = 0$ holds, which follows from \eqref{p2}$_4$. Thus
	\begin{multline*}
		\frac{1}{2}\Dt \left(\norm{v}^2_{L_2(\Omega)} + \norm{\omega}^2_{L_2(\Omega)}\right) + (\nu + \nu_r)\norm{\Rot v}^2_{L_2(\Omega)} + \alpha\norm{\Rot \omega}^2_{L_2(\Omega)} + (\alpha + \beta)\norm{\Div \omega}^2_{L_2(\Omega)}  \\
	 	+ 4\nu_r\norm{\omega}^2_{L_2(\Omega)} = 4\nu_r\int_{\Omega} \Rot v \cdot \omega\, \ud x + \int_{\Omega} f\cdot v\, \ud x + \int_{\Omega} g\cdot \omega\, \ud x.
	\end{multline*}
	Now we make necessary estimates. From Lemma \ref{l2} for $u = \omega$ it follows that 
	\begin{align*}
		\frac{\alpha}{c_{\Omega}}\norm{\omega}_{H^1(\Omega)}^2 &\leq \alpha\left(\norm{\Rot \omega}^2_{L_2(\Omega)} + \norm{\Div \omega}^2_{L_2(\Omega)}\right), \\
		\frac{\nu}{c_{\Omega}}\norm{v}^2_{H^1(\Omega)} &\leq \nu\norm{\Rot v}^2_{L_2(\Omega)}.
	\end{align*}
	By the H\"older and the Young inequalities we obtain
	\begin{align*}
		4\nu_r\int_{\Omega} \Rot v \cdot \omega\, \ud x &\leq 4\nu_r\epsilon_1\norm{\Rot v}^2_{L_2(\Omega)} + \frac{\nu_r}{\epsilon_1}\norm{\omega}^2_{L_2(\Omega)}, \\
		\int_{\Omega} f\cdot v\, \ud x &\leq \epsilon_2\norm{v}^2_{L_2(\Omega)} + \frac{1}{4\epsilon_2}\norm{f}^2_{L_2(\Omega)} \leq \epsilon_2 \norm{v}^2_{H^1(\Omega)} + \frac{1}{4\epsilon_2}\norm{f}^2_{L_2(\Omega)}, \\
		\int_{\Omega} g\cdot \omega\, \ud x &\leq \epsilon_3\norm{\omega}^2_{L_2(\Omega)} + \frac{1}{4\epsilon_3}\norm{g}^2_{L_2(\Omega)} \leq \epsilon_3 \norm{\omega}^2_{H^1(\Omega)} + \frac{1}{4\epsilon_3}\norm{g}^2_{L_2(\Omega)}.
	\end{align*}
	Now we set $\epsilon_1 = \frac{1}{4}$, $\epsilon_2 = \frac{\nu}{2c_{\Omega}}$, $\epsilon_3 = \frac{\alpha}{2c_{\Omega}}$. Since $\norm{v}_{L_2(\Omega)} \leq \norm{v}_{H^1(\Omega)}$ and $\norm{\omega}_{L_2(\Omega)} \leq \norm{\omega}_{H^1(\Omega)}$ we finally get
	\begin{equation*}
		\frac{1}{2}\Dt \left(\norm{v}^2_{L_2(\Omega)} + \norm{\omega}^2_{L_2(\Omega)}\right) + \frac{\nu}{2c_{\Omega}}\norm{v}^2_{L_2(\Omega)} + \frac{\alpha}{2c_{\Omega}}\norm{\omega}^2_{L_2(\Omega)} \\
	 	\leq \frac{c_{\Omega}}{2\nu}\norm{f}^2_{L_2(\Omega)} + \frac{c_{\Omega}}{2\alpha}\norm{g}^2_{L_2(\Omega)}.
	\end{equation*}
	With $\bar c_{\alpha,\nu,\Omega} = \frac{\min\{\alpha,\nu\}}{c_{\Omega}}$, the above inequality becomes
	\begin{equation*}
		\Dt \left(\left(\norm{v}^2_{L_2(\Omega)} + \norm{\omega}^2_{L_2(\Omega)}\right)e^{\bar c_{\alpha,\nu,\Omega}t}\right) \leq \frac{c_{\Omega}}{\nu} \norm{f}^2_{L_2(\Omega)} e^{\bar c_{\alpha,\nu,\Omega}t} + \frac{c_{\Omega}}{\alpha}\norm{g}^2_{L_2(\Omega)}e^{\bar c_{\alpha,\nu,\Omega}t}.
	\end{equation*}
	Integrating with respect to $t \in (t_0,t_1)$ yields
	\begin{multline*}
		\left(\norm{v(t)}^2_{L_2(\Omega)} + \norm{\omega(t)}^2_{L_2(\Omega)}\right)e^{\bar c_{\alpha,\nu,\Omega}t} \leq \frac{c_{\Omega}}{\nu} \int_{t_0}^t\norm{f(s)}^2_{L_2(\Omega)} e^{\bar c_{\alpha,\nu,\Omega}s}\, \ud s \\
		+ \frac{c_{\Omega}}{\alpha}\int_{t_0}^t\norm{g(s)}^2_{L_2(\Omega)}e^{\bar c_{\alpha,\nu,\Omega}s}\, \ud s + \left(\norm{v(t_0)}^2_{L_2(\Omega)} + \norm{\omega(t_0)}^2_{L_2(\Omega)}\right)e^{\bar c_{\alpha,\nu,\Omega}t_0}
	\end{multline*}
	or equivalently
	\begin{multline*}
		\norm{v(t)}^2_{L_2(\Omega)} + \norm{\omega(t)}^2_{L_2(\Omega)} \leq \frac{c_{\Omega}}{\nu} e^{-\bar c_{\alpha,\nu,\Omega}t}\int_{t_0}^t\norm{f(s)}^2_{L_2(\Omega)} e^{\bar c_{\alpha,\nu,\Omega}s}\, \ud s \\
		+ \frac{c_{\Omega}}{\alpha}e^{-\bar c_{\alpha,\nu,\Omega}t}\int_{t_0}^t\norm{g(s)}^2_{L_2(\Omega)}e^{\bar c_{\alpha,\nu,\Omega}s}\, \ud s + \left(\norm{v(t_0)}^2_{L_2(\Omega)} + \norm{\omega(t_0)}^2_{L_2(\Omega)}\right)e^{-\bar c_{\alpha,\nu,\Omega}(t - t_0)},
	\end{multline*}
	which proofs assertion $(\mathbf{A})$.

	Next we use the assumption on the external data in the above inequality
	\begin{multline*}
		\frac{c_{\Omega}}{\nu} e^{-\bar c_{\alpha,\nu,\Omega}t}\int_{t_0}^t\norm{f(s)}^2_{L_2(\Omega)} e^{\bar c_{\alpha,\nu,\Omega}s}\, \ud s + \frac{c_{\Omega}}{\alpha}e^{-\bar c_{\alpha,\nu,\Omega}t}\int_{t_0}^t\norm{g(s)}^2_{L_2(\Omega)}e^{\bar c_{\alpha,\nu,\Omega}s}\, \ud s \\
		\leq \frac{c_{\Omega}}{\nu} e^{-\bar c_{\alpha,\nu,\Omega}t}\norm{f(t_0)}^2_{L_2(\Omega)}\int_{t_0}^t e^{-(s - t_0)} e^{\bar c_{\alpha,\nu,\Omega}s}\, \ud s \\
		+ \frac{c_{\Omega}}{\alpha}e^{-\bar c_{\alpha,\nu,\Omega}t}\norm{g(t_0)}^2_{L_2(\Omega)}\int_{t_0}^te^{-(s - t_0)}e^{\bar c_{\alpha,\nu,\Omega}s}\, \ud s \\
		\leq \frac{c_{\Omega}}{\nu} \norm{f(t_0)}^2_{L_2(\Omega)}e^{-(t - t_0)} + \frac{c_{\Omega}}{\alpha} \norm{g(t_0)}^2_{L_2(\Omega)}e^{-(t - t_0)}.
	\end{multline*}
	Thus
	\begin{multline*}
		\norm{v(t)}^2_{L_2(\Omega)} + \norm{\omega(t)}^2_{L_2(\Omega)} \leq \frac{c_{\Omega}}{\nu} \norm{f(t_0)}^2_{L_2(\Omega)}e^{-(t - t_0)} + \frac{c_{\Omega}}{\alpha} \norm{g(t_0)}^2_{L_2(\Omega)}e^{-(t - t_0)} \\
		+ \left(\norm{v(t_0)}^2_{L_2(\Omega)} + \norm{\omega(t_0)}^2_{L_2(\Omega)}\right)e^{-\bar c_{\alpha,\nu,\Omega}(t - t_0)},
	\end{multline*}
	which is precisely assertion $(\mathbf{B})$ of lemma.
\end{proof}

\begin{lem}\label{lem22}
	Suppose that $v(t_0), \omega(t_0) \in H^1(\Omega)$. Assume that
	\begin{align*}
		&\norm{f(t)}_{L_2(\Omega)} \leq \norm{f(t_0)}_{L_2(\Omega)}e^{-(t - t_0)}, & & &\norm{f_{,x_3}(t)}_{L_2(\Omega)} &\leq \norm{f_{,x_3}(t_0)}_{L_2(\Omega)}e^{-(t - t_0)},\\
		&\norm{g(t)}_{L_2(\Omega)} \leq \norm{g(t_0)}_{L_2(\Omega)}e^{-(t - t_0)}, & & &\norm{g_{,x_3}(t)}_{L_2(\Omega)} &\leq \norm{g_{,x_3}(t_0)}_{L_2(\Omega)}e^{-(t - t_0)}
	\end{align*}
	for $t_0 \leq t \leq t_1$.	Then
	\begin{equation*}
		\norm{v(t_1)}_{H^1(\Omega)} + \norm{\omega(t_1)}_{H^1(\Omega)} \leq c_{\nu,\alpha,\beta,I,P,\Omega}(t_1)\left(\norm{v(t_0)}_{H^1(\Omega)} + \norm{\omega(t_0)}_{H^1(\Omega)}\right),
	\end{equation*}
	where
	\begin{equation*}
		\lim_{t_1 \to \infty} c_{\nu,\alpha,\beta,I,P,\Omega}(t_1) = 0.
	\end{equation*}
\end{lem}

\begin{rem}
	We have assumed an exponential decay with respect to time on the external data and their derivative along the axis of the cylinder. The reason underlying  this assumption follows from the necessity to control the amount of the energy supplied to the system. If we consider the global in time existence, the supplied energy must be balanced by the rate of its lost due to friction, otherwise the system blows-up. However, we emphasize that this exponential decay with respect to time for the external data has only a local character. If $t_1$ denotes the end of the time interval under consideration, then
	\begin{equation*}
		\lim_{t \to t_1^-}\norm{f(t)}_{L_2(\Omega)} \neq \lim_{t\to t_1^+}\norm{f(t)}_{L_2(\Omega)}
	\end{equation*}
and analogously for the other external data. The right-hand side limit can be even much larger than the left-hand side limit. 

Note, that alternatively we may attempt to analyze carefully the direct correspondence between the energy input and loss which is likely to result in different assumption on the external data. However, it is beyond the scope of our study. For further reading we refer the reader to \cite{land}.

Note also, that if we had the Poincar\'e inequality for $f$ and $g$ with respect to $x_3$, we would be able to relax the assumption on the exponential decay of the external data and limit our considerations only to their rate of change. 
\end{rem}

\begin{proof}
	First we multiply \eqref{p1}$_{1,2}$ by $-\alpha\triangle v$ and $-\alpha\triangle \omega - \beta\nabla\Div \omega$ respectively and integrate over $\Omega$. It yields
	\begin{subequations}\label{eq22}
	\begin{multline}
		-\alpha\int_{\Omega} v_{,t}\cdot \triangle v\, \ud x + (\nu + \nu_r)\alpha\norm{\triangle v}_{L_2(\Omega)}^2 - \alpha\int_{\Omega} v\cdot \nabla v \cdot \triangle v\, \ud x - \alpha\int_{\Omega} \nabla p\cdot \triangle v\, \ud x \\
		= -2\nu_r\alpha\int_\Omega \Rot \omega \cdot \triangle v\, \ud x - \alpha\int_\Omega f\cdot \triangle v\, \ud x
	\end{multline}
	and
	\begin{multline}
		-\int_{\Omega} \omega_{,t} \cdot (\alpha\triangle \omega + \beta\nabla\Div\omega)\, \ud x + \int_{\Omega} \left(\alpha\triangle \omega+ \beta\nabla \Div\omega\right)^2\, \ud x \\
		- 4\nu_r\int_\Omega \omega\cdot (\alpha\triangle \omega + \beta\nabla\Div\omega)\, \ud x = - 2\nu_r\int_\Omega \Rot v \cdot (\alpha\triangle \omega + \beta\nabla\Div\omega)\, \ud x \\
		- \int_\Omega g\cdot (\alpha\triangle \omega + \beta\nabla\Div\omega)\, \ud x + \int_{\Omega} v\cdot \nabla \omega \cdot (\alpha\triangle\omega + \beta\nabla\Div\omega) \, \ud x.
	\end{multline}
	\end{subequations}
	By application of Lemma \ref{l1} we see that
	\begin{align*}
		&-\alpha\int_{\Omega} v_{,t}\cdot \triangle v\, \ud x  = \alpha\int_{\Omega} \Rot v_{,t}\cdot \Rot v\, \ud x + \alpha\int_{S} \Rot v\times n\cdot v_{,t}\, \ud S = \frac{1}{2}\Dt \alpha\int_{\Omega} \abs{\Rot v}^2\, \ud x, \\
		&-\alpha\int_{\Omega} \nabla p \cdot \triangle v\, \ud x = \alpha\int_{\Omega} \Rot \nabla p \cdot \Rot v\, \ud x + \alpha\int_S \Rot v\times n \cdot \nabla p\, \ud S = 0,
	\end{align*}
	where we used the boundary conditions \eqref{p2}$_2$.
	From the vector identity $-\triangle \omega = \Rot\Rot \omega - \nabla\Div \omega$ and Lemma \ref{l1} it follows that
	\begin{multline*}
		-\int_{\Omega} \omega_{,t} \cdot \triangle \omega\, \ud x = \int_{\Omega} \omega_{,t} \cdot \Rot\Rot \omega - \nabla \Div \omega\, \ud x = \frac{1}{2} \Dt \int_{\Omega} \abs{\Rot \omega}^2\, \ud x - \int_S \omega_{,t}\times n \cdot \Rot \omega\, \ud S \\
		+ \int_{\Omega}\omega_{,t} \cdot \nabla\Div \omega\, \ud x = \frac{1}{2} \Dt \int_{\Omega} \abs{\Rot \omega}^2 + \abs{\Div \omega}^2\, \ud x - \int_S \Div \omega \left(\omega_{,t} \cdot n\right)\, \ud S.
	\end{multline*}
	The boundary conditions \eqref{p2}$_4$ imply that $\omega'_{,x'}\vert_{S_2} = 0$. Thus, $\Div \omega\vert_{S_2} = 0$. Since $\omega\vert_{S_1} = 0 \Rightarrow \omega_{,t}\vert_{S_1} = 0$ and $\omega_{,t}\times n = [-\omega_{2,t},\omega_{1,t},0] = 0$ we see that the boundary integrals vanish. Therefore
	\begin{equation*}
		-\int_{\Omega} \omega_{,t}\cdot\left(\alpha\triangle \omega + \beta\nabla\Div \omega\right)\, \ud x = \frac{1}{2}\Dt\int_{\Omega} \alpha\abs{\Rot \omega}^2 + (\alpha + \beta)\abs{\Div \omega}^2\, \ud x.
	\end{equation*}
	In the same manner we get
	\begin{equation*}
		-\int_{\Omega} \omega \cdot \left(\alpha\triangle \omega + \beta\nabla\Div \omega\right)\, \ud x = \int_{\Omega} \alpha\abs{\Rot \omega}^2 + (\alpha + \beta)\abs{\Div \omega}^2\, \ud x.
	\end{equation*}
	Moving to the first term on the right-hand side we encounter a problem with the integral
	\begin{equation*}
		\beta\int_{\Omega} \Rot v\cdot \nabla\Div \omega\, \ud x.
	\end{equation*}
	We easily see that the application of Lemma \ref{l1} or integration by parts lead to a boundary integral, either
	\begin{equation*}
		\beta\int_S v\times n \cdot \nabla \Div \omega\, \ud S \qquad \text{or} \qquad \beta\int_S \Div \omega \Rot v \cdot n\, \ud S.
	\end{equation*}
	We are not able to compute these integrals, because the boundary conditions \eqref{p2} yield insufficient information. The first integral contains second-order derivatives of $\omega$, which makes it impossible to estimate by the data in any suitable norm. For the second we apply trace theorem and interpolation inequality, which we shall present later. For now we only write
	\begin{multline*}
		-\int_{\Omega} \Rot v \cdot (\alpha\triangle \omega + \beta\nabla\Div\omega)\, \ud x = \alpha\int_{\Omega} \Rot v \cdot \Rot \Rot \omega\, \ud x - (\alpha + \beta)\int_{\Omega} \Rot v \cdot \nabla \Div \omega\, \ud x\\
		= -\alpha\int_{\Omega}\triangle v \cdot \Rot \omega\, \ud x - \alpha\int_S \Rot v \times n \cdot \Rot \omega\, \ud S - (\alpha + \beta)\int_S \Div \omega\Rot v \cdot n\, \ud S \\
		= -\alpha\int_{\Omega}\triangle v \cdot \Rot \omega\, \ud x - (\alpha + \beta)\int_S \Div \omega\Rot v \cdot n\, \ud S.
	\end{multline*}
	Finally, we rewrite \eqref{eq22} in the following form
	\begin{multline}\label{eq23}
		\frac{1}{2}\Dt\left( \int_{\Omega} \alpha\abs{\Rot v}^2 + \alpha\abs{\Rot \omega}^2 + (\alpha + \beta)\abs{\Div \omega}^2\, \ud x\right) + (\nu + \nu_r)\alpha\norm{\triangle v}_{L_2(\Omega)}^2 \\
		+ \norm{\alpha \triangle \omega + \beta\nabla\Div\omega}_{L_2(\Omega)}^2 + 4\nu_r \alpha\int_{\Omega}\abs{\Rot \omega}^2 + (\alpha + \beta)\abs{\Div \omega}^2\, \ud x \\
		= -4\nu_r \alpha\int_{\Omega} \Rot \omega \cdot \triangle v\, \ud x - 2\nu_r(\alpha + \beta) \int_S \Div \omega\Rot v\cdot n\,\ud S + \int_{\Omega} v\cdot \nabla v\cdot \triangle v\, \ud x \\
		+ \int_{\Omega} v\cdot \nabla \omega \cdot (\alpha\triangle \omega + \beta\nabla\Div\omega)\, \ud x - \int_{\Omega} f\cdot \triangle v\,\ud x - \int_{\Omega} g\cdot (\alpha\triangle \omega + \beta\nabla\Div \omega)\, \ud x \\
		=: \sum_{k = 1}^6 I_k.
	\end{multline}
	We shall estimate $I_k$ by application of the H\"older and the Young inequalities. Two first one is obvious:
	\begin{equation*}
		I_1 \leq 4\nu_r \alpha \norm{\Rot \omega}_{L_2(\Omega)}\norm{\triangle v}_{L_2(\Omega)} \leq 4\nu_r\alpha \epsilon_1 \norm{\triangle v}^2_{L_2(\Omega)} + \frac{\nu_r\alpha}{\epsilon_1} \norm{\Rot \omega}^2_{L_2(\Omega)}.
	\end{equation*}
	For $I_2$ more work is required. From the trace theorem it follows that for $\epsilon > 0$
	\begin{multline*}
		I_2 \leq 2\nu_r(\alpha + \beta)\norm{\Div \omega}_{L_2(S)}\norm{\Rot v \cdot n}_{L_2(S)} \leq \nu_r(\alpha + \beta)\norm{\Div \omega}^2_{L_2(S)} + \nu_r(\alpha + \beta)\norm{\Rot v}^2_{L_2(S)} \\
		\leq c_{\Omega}\nu_r(\alpha + \beta) \norm{\Div \omega}^2_{H^{\frac{1}{2} - \epsilon}(\Omega)} + c_{\Omega}(\alpha + \beta)\norm{\Rot v}_{H^{\frac{1}{2} - \epsilon}(\Omega)}^2.
	\end{multline*}
	Now we use an inequality of Gagliardo-Nirenberg type (see \cite[Ch. 1, Rem. 1.2.1]{chol})
	\begin{align*}
		\norm{\Div \omega}^2_{H^{\frac{1}{2} - \epsilon}(\Omega)} &\leq \norm{\omega}_{H^2(\Omega)}^{2\theta}\norm{\omega}_{L_2(\Omega)}^{2(1 - \theta)}, \\ 
		\norm{\Rot v}_{H^{\frac{1}{2} - \epsilon}(\Omega)}^2 &\leq \norm{v}_{H^2(\Omega)}^{2\theta}\norm{v}_{L_2(\Omega)}^{2(1 - \theta)},
	\end{align*}
	where $\theta$ satisfies
	\begin{equation*}
		\frac{3}{2} - \epsilon - \frac{3}{2} \leq (1 - \theta)\left(-\frac{3}{2}\right) + \theta\left(2 - \frac{3}{2}\right) \quad \Leftrightarrow \quad \frac{3}{2} - \epsilon \leq 2\theta,
	\end{equation*}
	which allows us to set $\theta = \frac{3}{4}$. Thus, by the Young inequality with $\epsilon$
	\begin{align*}
		\norm{\Div \omega}^2_{H^{\frac{1}{2} - \epsilon}(\Omega)} &\leq \norm{\omega}_{H^2(\Omega)}^{\frac{3}{2}}\norm{\omega}_{L_2(\Omega)}^{\frac{1}{2}} \leq \epsilon_{21}\norm{\omega}^2_{H^2(\Omega)} + \frac{1}{4\epsilon_{21}}\norm{\omega}^2_{L_2(\Omega)} \\
		\norm{\Rot v}_{H^{\frac{1}{2} - \epsilon}(\Omega)}^2 &\leq \norm{v}_{H^2(\Omega)}^{\frac{3}{2}}\norm{v}_{L_2(\Omega)}^{\frac{1}{2}} \leq \epsilon_{22}\norm{v}^2_{H^2(\Omega)} + \frac{1}{4\epsilon_{22}}\norm{v}^2_{L_2(\Omega)}.
	\end{align*}
	Finally
	\begin{multline*}
		I_2 \leq c_{\Omega}\nu_r(\alpha + \beta)\epsilon_{21}\norm{\omega}^2_{H^2(\Omega)} + \frac{c_{\Omega}\nu_r(\alpha + \beta)}{4\epsilon_{21}}\norm{\omega}^2_{L_2(\Omega)} \\
		+ c_{\Omega}\nu_r(\alpha + \beta)\epsilon_{22}\norm{v}^2_{H^2(\Omega)} + \frac{c_{\Omega}\nu_r(\alpha + \beta)}{4\epsilon_{22}}\norm{v}^2_{L_2(\Omega)}.
	\end{multline*}
	For the nonlinear terms we have
	\begin{equation*}
		I_3 \leq \norm{\triangle v}_{L_2(\Omega)}\norm{\nabla v}_{L_2(\Omega)}\norm{v}_{L_{\infty}(\Omega)} \leq \epsilon_3 \norm{\triangle v}^2_{L_2(\Omega)} + \frac{1}{4\epsilon_3} \norm{\nabla v}^2_{L_2(\Omega)}\norm{v}^2_{L_{\infty}(\Omega)}
	\end{equation*}
	and
	\begin{multline*}
		I_4 \leq \norm{\alpha\triangle \omega + \beta\nabla\Div\omega}_{L_2(\Omega)}\norm{\nabla \omega}_{L_2(\Omega)}\norm{v}_{L_{\infty}(\Omega)} \\
		\leq \epsilon_4 (\alpha + \beta)\norm{\omega}^2_{H^2(\Omega)} + \frac{\alpha + \beta}{4\epsilon_4} \norm{\nabla \omega}^2_{L_2(\Omega)}\norm{v}^2_{L_{\infty}(\Omega)}.
	\end{multline*}
	The two last terms are estimated in a standard way
	\begin{align*}
		I_5 &\leq \epsilon_5 \norm{\triangle v}^2_{L_2(\Omega)} + \frac{1}{4\epsilon_5}\norm{f}^2_{L_2(\Omega)}, \\
		I_6 &\leq \epsilon_6 (\alpha + \beta)\norm{\omega}^2_{H^2(\Omega)} + \frac{\alpha + \beta}{4\epsilon_6}\norm{g}^2_{L_2(\Omega)}.
	\end{align*}
	Before we chose $\epsilon_i$, $i = 1,\ldots,6$, we justify the following inequality
	\begin{equation}\label{eq110}
		\norm{v}^2_{H^2(\Omega)} \leq c_{\Omega}\norm{\triangle v}^2_{L_2(\Omega)}.
	\end{equation}
	If we put $u = \Rot v$ (we see that $u\times n\vert_S = 0$) in Lemma \ref{l1} then we get
	\begin{equation*}
		\norm{\Rot v}_{H^1(\Omega)}^2 \leq c_{\Omega}\norm{\Rot\Rot v}^2_{L_2(\Omega)} = c_{\Omega}\norm{\triangle v}^2_{L_2(\Omega)}.
	\end{equation*}
	On the other hand, let now $u = v$ (we see that $u \cdot n\vert_S = 0$) in Lemma \ref{l1} and use the above inequality. Then
	\begin{equation*}
		\norm{v}_{H^2(\Omega)}^2 \leq c_{\Omega} \norm{\Rot v}^2_{H^1(\Omega)} \leq c_{\Omega}\norm{\triangle v}^2_{L_2(\Omega)}.
	\end{equation*}

	We set 
	\begin{align*}
		&\epsilon_1 = \frac{1}{4\alpha}, \\
		&\epsilon_{21}c_{\Omega}\nu_r(\alpha + \beta) = \epsilon_4(\alpha + \beta) = \epsilon_6(\alpha + \beta) = \frac{1}{6(c_{\Omega}(\alpha + \beta) + 1)}, \\
		&\epsilon_{22}c_{\Omega}^2\nu_r(\alpha + \beta) = \epsilon_3 = \epsilon_5 = \frac{\nu\alpha}{6}.
	\end{align*}
	Thus, we get from \eqref{eq23} that
	\begin{multline}\label{eq24}
		\frac{1}{2}\left(\Dt \int_{\Omega} \alpha\abs{\Rot v}^2 + \alpha\abs{\Rot \omega}^2 + (\alpha + \beta)\abs{\Div \omega}^2\, \ud x\right) + \frac{\nu\alpha}{2c_{\Omega}}\norm{v}_{H^2(\Omega)}^2 \\
		+ \frac{1}{2(c_{\Omega}(\alpha + \beta) + 1)} \norm{\omega}_{H^2(\Omega)}^2  \\
		\leq \frac{3c_{\Omega}^2\nu_r^2(\alpha + \beta)^2(c_{\Omega}(\alpha + \beta) + 1)}{2}\norm{\omega}^2_{L_2(\Omega)} + \frac{3c_{\Omega}^3\nu_r^2(\alpha + \beta)^2}{2\nu\alpha}\norm{v}^2_{L_2(\Omega)}\\
		+ \frac{3}{2\nu\alpha} \norm{\nabla v}^2_{L_2(\Omega)}\norm{v}^2_{L_{\infty}(\Omega)} + \frac{3(c_{\Omega}(\alpha + \beta) + 1)(\alpha + \beta)^2}{2} \norm{\nabla \omega}^2_{L_2(\Omega)}\norm{v}^2_{L_{\infty}(\Omega)} \\
		+ \frac{3}{2\nu\alpha}\norm{f}^2_{L_2(\Omega)} + \frac{3(c_{\Omega}(\alpha + \beta) + 1)(\alpha + \beta)^2}{2}\norm{g}^2_{L_2(\Omega)}.
	\end{multline}
	where we use \eqref{eq110} at the end. The next step is to estimate the $H^2(\Omega)$-norms on the left-hand side from below by $L_2(\Omega)$-norms of $\Rot v$ and $\Rot \omega + \Div \omega$. For any $u \in H^2(\Omega)$ we have 
	\begin{equation*}
		\frac{1}{2}\left(\norm{\Rot u}^2_{L_2(\Omega)} + \norm{\Div u}^2_{L_2(\Omega)}\right) \leq \norm{u}^2_{H^1(\Omega)} \leq \norm{u}^2_{H^2(\Omega)}.
	\end{equation*}
	Let
	\begin{equation*}
		c_1 = \min\left\{ \frac{\nu\alpha}{c_{\Omega}}, \frac{1}{(c_{\Omega}(\alpha + \beta) + 1)}\right\}.
	\end{equation*}
	Then
	\begin{multline}\label{eq27}
		 \frac{\nu\alpha}{2c_{\Omega}}\norm{v}_{H^2(\Omega)}^2 + \frac{1}{2(c_{\Omega}(\alpha + \beta) + 1)} \norm{\omega}_{H^2(\Omega)}^2 \geq \frac{c_1}{2} \left(\norm{v}^2_{H^2(\Omega)} + \norm{\omega}^2_{H^2(\Omega)}\right) \\
		\geq \frac{c_1}{4}\left(\norm{\Rot v}^2_{L_2(\Omega)} + \norm{\Rot \omega}^2_{L_2(\Omega)} + \norm{\Div \omega}^2_{L_2(\Omega)}\right) \\
		\geq \frac{c_1}{4(\alpha + \beta) + 2\delta}\left(\alpha\norm{\Rot v}^2_{L_2(\Omega)} + \beta\norm{\Rot \omega}^2_{L_2(\Omega)} + (\alpha + \beta)\norm{\Div \omega}^2_{L_2(\Omega)}\right),
	\end{multline}
	where $\delta \geq 1$ will be chosen later. 
	On the other hand, in view of Lemma \ref{l2} we already know that
	\begin{align*}
		\norm{\nabla v}^2_{L_2(\Omega)} &\leq \norm{v}^2_{H^1(\Omega)} \leq c_{\Omega}\norm{\Rot v}^2_{L_2(\Omega)}, \\
		\norm{\nabla \omega}^2_{L_2(\Omega)} &\leq \norm{\omega}^2_{H^1(\Omega)} \leq c_{\Omega}\left(\norm{\Rot \omega}^2_{L_2(\Omega)} + \norm{\Div \omega}^2_{L_2(\Omega)}\right).
	\end{align*}
	Let
	\begin{align*}
		c_2 &= \max\left\{\frac{3}{\nu\alpha},3(c_{\Omega}(\alpha + \beta) + 1)(\alpha + \beta)^2\right\}, \\
		c_3 &= \max\left\{3c_{\Omega}^2\nu_r^2(\alpha + \beta)^2(c_{\Omega}(\alpha + \beta) + 1),\frac{3c_{\Omega}^3\nu_r^2(\alpha + \beta)^2}{\nu\alpha},\frac{3}{\nu\alpha},3(c_{\Omega}(\alpha + \beta) + 1)(\alpha + \beta)^2\right\}.
	\end{align*}
	Denote
	\begin{equation*}
		X(t) = \alpha\norm{\Rot v(t)}^2_{L_2(\Omega)} + \alpha\norm{\Rot \omega(t)}^2_{L_2(\Omega)} + (\alpha + \beta)\norm{\Div \omega(t)}^2_{L_2(\Omega)}.
	\end{equation*}
	Then, combining \eqref{eq27} with \eqref{eq24} gives
	\begin{multline*}
		\frac{1}{2}\Dt X(t) + \frac{c_1}{4(\alpha + \beta) + 2\delta} X(t) \\
		\leq \frac{c_2}{2} \norm{v}^2_{L_{\infty}(\Omega)} X(t) + \frac{c_3}{2}\left(\norm{f}^2_{L_2(\Omega)} + \norm{g}^2_{L_2(\Omega)} + \norm{v}^2_{L_2(\Omega)} + \norm{\omega}^2_{L_2(\Omega)}\right),
	\end{multline*}
	which is equivalent to
	\begin{multline*}
		\Dt X(t) + \left(\frac{c_1}{2(\alpha + \beta) + \delta} - c_2\norm{v}^2_{L_{\infty}(\Omega)}\right)X(t) \\
		\leq c_3\left(\norm{f}^2_{L_2(\Omega)} + \norm{g}^2_{L_2(\Omega)} + \norm{v}^2_{L_2(\Omega)} + \norm{\omega}^2_{L_2(\Omega)}\right).
	\end{multline*}
	We may also write
	\begin{multline*}
		\Dt\left(X(t)e^{\frac{c_1}{2(\alpha + \beta) + \delta}t - c_2\int_{t_0}^t\norm{v(s)}^2_{L_{\infty}(\Omega)}\,\ud s}\right) \\
			\leq c_3\left(\norm{f}^2_{L_2(\Omega)} + \norm{g}^2_{L_2(\Omega)} + \norm{v}^2_{L_2(\Omega)} + \norm{\omega}^2_{L_2(\Omega)}\right)e^{\frac{c_1}{2(\alpha + \beta) + \delta}t - c_2\int_{t_0}^t\norm{v(s)}^2_{L_{\infty}(\Omega)}\,\ud s}.
	\end{multline*}
	Integrating with respect to $t \in (t_0,t_1)$ yields
	\begin{multline*}
		X(t_1)e^{\frac{c_1}{2(\alpha + \beta) + \delta}t_1 - c_2\int_{t_0}^{t_1}\norm{v(s)}^2_{L_{\infty}(\Omega)}\,\ud s} \\
		\leq c_3\int_{t_0}^{t_1}\left(\norm{f(t)}^2_{L_2(\Omega)} + \norm{g(t)}^2_{L_2(\Omega)} + \norm{v(t)}^2_{L_2(\Omega)} + \norm{\omega(t)}^2_{L_2(\Omega)}\right)\\
		\cdot e^{\frac{c_1}{2(\alpha + \beta) + \delta}t - c_2\int_{t_0}^t\norm{v(s)}^2_{L_{\infty}(\Omega)}\,\ud s}\, \ud t + X(t_0)e^{\frac{c_1}{2(\alpha + \beta) + \delta}t_0}.
	\end{multline*}
	Next we divide both sides by $e^{\frac{c_1}{2(\alpha + \beta) + \delta}t_1 - c_2\int_{t_0}^{t_1}\norm{v(s)}^2_{L_{\infty}(\Omega)}\,\ud s}$. It gives
	\begin{multline}\label{eq29}
		X(t_1) \leq c_3e^{-\frac{c_1}{2(\alpha + \beta) + \delta}t_1 + c_2\int_{t_0}^{t_1}\norm{v(s)}^2_{L_{\infty}(\Omega)}\,\ud s}\\
		\cdot \int_{t_0}^{t_1}\left(\norm{f(t)}^2_{L_2(\Omega)} + \norm{g(t)}^2_{L_2(\Omega)} + \norm{v(t)}^2_{L_2(\Omega)} + \norm{\omega(t)}^2_{L_2(\Omega)}\right)\cdot e^{\frac{c_1}{2(\alpha + \beta) + \delta}t - c_2\int_{t_0}^t\norm{v(s)}^2_{L_{\infty}(\Omega)}\,\ud s}\, \ud t \\
		+ X(t_0)e^{-\frac{c_1}{2(\alpha + \beta) + \delta}(t_1 - t_0) + c_2\int_{t_0}^{t_1}\norm{v(s)}^2_{L_{\infty}(\Omega)}\,\ud s}.
	\end{multline}
	Consider the integral with respect to $t$ on the right-hand side. Since $\norm{v(s)}^2_{L_{\infty}(\Omega)}$ is non-negative, we can write
	\begin{equation*}
		e^{\frac{c_1}{2(\alpha + \beta) + \delta}t - c_2 \int_{t_0}^t\norm{v(s)}^2_{L_{\infty}(\Omega)}} \leq e^{\frac{c_1}{2(\alpha + \beta) + \delta}t}.
	\end{equation*}
	By assumption and assertion $(\mathbf{B})$ from Lemma \ref{lem28} we see that
	\begin{multline}\label{eq28}
		c_3e^{-\frac{c_1}{2(\alpha + \beta) + \delta}t_1 + c_2\int_{t_0}^{t_1}\norm{v(s)}^2_{L_{\infty}(\Omega)}\,\ud s}\\
		\cdot \int_{t_0}^{t_1}\left(\norm{f(t)}^2_{L_2(\Omega)} + \norm{g(t)}^2_{L_2(\Omega)} + \norm{v(t)}^2_{L_2(\Omega)} + \norm{\omega(t)}^2_{L_2(\Omega)}\right)\cdot e^{\frac{c_1}{2(\alpha + \beta) + \delta}t - c_2\int_{t_0}^t\norm{v(s)}^2_{L_{\infty}(\Omega)}\,\ud s}\, \ud t \\
		\leq c_3e^{-\frac{c_1}{2(\alpha + \beta) + \delta}t_1 + c_2\int_{t_0}^{t_1}\norm{v(s)}^2_{L_{\infty}(\Omega)}\,\ud s}\\
		\cdot \int_{t_0}^{t_1}\left(\norm{f(t)}^2_{L_2(\Omega)} + \norm{g(t)}^2_{L_2(\Omega)} + \norm{v(t)}^2_{L_2(\Omega)} + \norm{\omega(t)}^2_{L_2(\Omega)}\right)\cdot e^{\frac{c_1}{2(\alpha + \beta) + \delta}t}\, \ud t \\
		\leq c_3e^{-\frac{c_1}{2(\alpha + \beta) + \delta}t_1 + c_2\int_{t_0}^{t_1}\norm{v(s)}^2_{L_{\infty}(\Omega)}\,\ud s} \\
		\cdot \left( \left(1 + \frac{c_{\Omega}}{\nu}\right)\norm{f(t_0)}^2_{L_2(\Omega)} + \left(1 + \frac{c_{\Omega}}{\alpha}\right)\norm{g(t_0)}^2_{L_2(\Omega)} + \norm{v(t_0)}^2_{L_2(\Omega)} + \norm{\omega(t_0)}^2_{L_2(\Omega)} \right) \\
		\cdot \int_{t_0}^{t_1} e^{\frac{c_1}{2(\alpha + \beta) + \delta}t} \cdot e^{-\min\left\{1,{\frac{\min\{\alpha,\nu\}}{c_{\Omega}}}\right\}(t - t_0)}\, \ud t.
	\end{multline}
	Now we chose $\delta$. Two cases may occur:
	\begin{description}
		\item[1. $\alpha < c_{\Omega} \vee \nu < c_{\Omega}$.] In this case
			\begin{equation*}
				\min\left\{1,\frac{\min\{\alpha,\nu\}}{c_{\Omega}}\right\} = \frac{\min\{\alpha,\nu\}}{c_{\Omega}},
			\end{equation*}
			so we chose such $\delta$ that it satisfies
			\begin{equation}\label{eq30}
				\frac{c_1}{2(\alpha + \beta) + \delta} - \frac{\min\{\alpha,\nu\}}{c_{\Omega}} < 0 \quad \Leftrightarrow \quad \frac{c_1c_{\Omega} - 2\min\{\alpha,\nu\}(\alpha + \beta)}{\min\{\alpha,\nu\}} < \delta.
			\end{equation}
			Then 
			\begin{multline*}
				 c_3e^{-\frac{c_1}{2(\alpha + \beta) + \delta}t_1 + c_2\int_{t_0}^{t_1}\norm{v(s)}^2_{L_{\infty}(\Omega)}\,\ud s}\int_{t_0}^{t_1}e^{\frac{c_1}{2(\alpha + \beta) + \delta}t} \cdot e^{-\min\left\{1,{\frac{\min\{\alpha,\nu\}}{c_{\Omega}}}\right\}(t - t_0)}\, \ud t \\
				= c_3e^{-\frac{c_1}{2(\alpha + \beta) + \delta}t_1 + c_2\int_{t_0}^{t_1}\norm{v(s)}^2_{L_{\infty}(\Omega)}\,\ud s + \frac{\min\{\alpha,\nu\}}{c_{\Omega}}t_0} \frac{1}{\frac{c_1}{2(\alpha + \beta) + \delta} - \frac{\min\{\alpha,\nu\}}{c_{\Omega}}}e^{\left(\frac{c_1}{2(\alpha + \beta) + \delta} - \frac{\min\{\alpha,\nu\}}{c_{\Omega}}\right)t}\Bigg\vert_{t_0}^{t_1} \\
				\leq c_3e^{-\frac{c_1}{2(\alpha + \beta) + \delta}(t_1 - t_0) + c_2\int_{t_0}^{t_1}\norm{v(s)}^2_{L_{\infty}(\Omega)}\,\ud s} \frac{1}{\frac{\min\{\alpha,\nu\}}{c_{\Omega}} - \frac{c_1}{2(\alpha + \beta) + \delta}}.
			\end{multline*}
		\item[2. $\alpha > c_{\Omega} \wedge \nu > c_{\Omega}$.] Now we have
			\begin{equation*}
				\min\left\{1,\frac{\min\{\alpha,\nu\}}{c_{\Omega}}\right\} = 1.
			\end{equation*}
			Since $c_{\Omega}(\alpha + \beta) + 1 > 1$ we see that $c_1 < 1$. It is enough to set $\delta = 1$. Then
			\begin{equation*}
				\frac{c_1}{2(\alpha + \beta) + 1} - 1 < 0
			\end{equation*}
			and
			\begin{multline*}
				 c_3e^{-\frac{c_1}{2(\alpha + \beta) + \delta}t_1 + c_2\int_{t_0}^{t_1}\norm{v(s)}^2_{L_{\infty}(\Omega)}\,\ud s}\int_{t_0}^{t_1}e^{\frac{c_1}{2(\alpha + \beta) + \delta}t} \cdot e^{-\min\left\{1,{\frac{\min\{\alpha,\nu\}}{c_{\Omega}}}\right\}(t - t_0)}\, \ud t \\
				= c_3e^{-\frac{c_1}{2(\alpha + \beta) + 1}t_1 + c_2\int_{t_0}^{t_1}\norm{v(s)}^2_{L_{\infty}(\Omega)}\,\ud s + t_0} \frac{1}{\frac{c_1}{2(\alpha + \beta) + 1} - 1}e^{\left(\frac{c_1}{2(\alpha + \beta) + 1} - 1\right)t}\bigg\vert_{t_0}^{t_1} \\
				\leq c_3e^{-\frac{c_1}{2(\alpha + \beta) + 1}(t_1 - t_0) + c_2\int_{t_0}^{t_1}\norm{v(s)}^2_{L_{\infty}(\Omega)}\,\ud s} \frac{1}{1 - \frac{c_1}{2(\alpha + \beta) + 1}}.
			\end{multline*}
	\end{description}
	In both cases we have obtained an estimate of the form
	\begin{equation*}
		c_{\nu,\alpha,\beta,\Omega}c_3e^{-\frac{c_1}{2(\alpha + \beta) + \delta}(t_1 - t_0) + c_2\int_{t_0}^{t_1}\norm{v(s)}^2_{L_{\infty}(\Omega)}\,\ud s},
	\end{equation*}
	where either $\delta = 1$ or $\delta$ satisfies \eqref{eq30}. Putting the above estimate into \eqref{eq28}, we get from \eqref{eq29} that
	\begin{multline*}
		X(t_1) \leq c_{\nu,\alpha,\beta,\Omega}c_3e^{-\frac{c_1}{2(\alpha + \beta) + \delta}(t_1 - t_0) + c_2\int_{t_0}^{t_1}\norm{v(s)}^2_{L_{\infty}(\Omega)}\,\ud s}\\
		\cdot c_{\nu,\alpha,\Omega}\left(\norm{f(t_0)}^2_{L_2(\Omega)} + \norm{g(t_0)}^2_{L_2(\Omega)} + \norm{v(t_0)}^2_{L_2(\Omega)} + \norm{\omega(t_0)}^2_{L_2(\Omega)}\right) \\
		+ X(t_0)e^{-\frac{c_1}{2(\alpha + \beta) + \delta}(t_1 - t_0) + c_2\int_{t_0}^{t_1}\norm{v(s)}^2_{L_{\infty}(\Omega)}\,\ud s}.
	\end{multline*}
	In view of Lemma \ref{l2} and an obvious inequality $\alpha < \alpha + \beta$ we can rewrite the above estimate in the form
	\begin{multline*}
		\frac{\alpha}{c_{\Omega}} \left(\norm{v(t_1)}^2_{H^1(\Omega)} + \norm{\omega(t_1)}^2_{H^1(\Omega)}\right) \leq c_{\nu,\alpha,\beta,\Omega}c_3e^{-\frac{c_1}{2(\alpha + \beta) + \delta}(t_1 - t_0) + c_2\int_{t_0}^{t_1}\norm{v(s)}^2_{L_{\infty}(\Omega)}\,\ud s}\\
		\cdot c_{\nu,\alpha,\Omega}\left(\norm{f(t_0)}^2_{L_2(\Omega)} + \norm{g(t_0)}^2_{L_2(\Omega)} + \norm{v(t_0)}^2_{L_2(\Omega)} + \norm{\omega(t_0)}^2_{L_2(\Omega)}\right) \\
		+ 2(\alpha + \beta)\left(\norm{v(t_0)}^2_{H^1(\Omega)} + \norm{\omega(t_0)}^2_{H^1(\Omega)}\right)e^{-\frac{c_1}{2(\alpha + \beta) + \delta}(t_1 - t_0) + c_2\int_{t_0}^{t_1}\norm{v(s)}^2_{L_{\infty}(\Omega)}\,\ud s}.
	\end{multline*}	
	Next observe that since $H^2(\Omega) \hookrightarrow L_\infty(\Omega)$, we have 
	\begin{equation*}
		\norm{v(s)}^2_{L_{\infty}(\Omega)} \leq c_I\norm{v(s)}^2_{H^2(\Omega)}.
	\end{equation*}
	Integrating this inequality with respect to $t \in (t_0,t_1)$ and utilizing the estimate from Theorem \ref{t1} we obtain
	\begin{multline*}
		\int_{t_0}^{t_1}\norm{v(s)}^2_{L_{\infty}(\Omega)}\, \ud s \leq c_I\norm{v}^2_{L_2(t_0,t_1;H^2(\Omega))} \leq c_I\norm{v}^2_{W^{2,1}_2(\Omega^{t_1})} \\
		\leq c_{\alpha,\nu,\nu_r,I,P,\Omega}\Big(E^3_{v,\omega}(t_1) + E^3_{h,\theta}(t_1) + \norm{f'}^3_{L_2(\Omega^{t_1})} + \norm{v(t_0)}^3_{H^1(\Omega)} + E_{v,\omega}(t_1) + \norm{v(t_0)}_{H^1(\Omega)}\Big)^2.
	\end{multline*}
	By assumption on $f$, $g$, their derivative with respect to $x_3$ we see that 
	\begin{multline*}
		E_{v,\omega}(t_1) + E_{h,\theta}(t_1)\leq c_{\Omega}\Big(\norm{v(t_0)}_{H^1(\Omega)} + \norm{\omega(t_0)}_{H^1(\Omega)} + \norm{f(t_0)}_{L_2(\Omega)} + \norm{g(t_0)}_{L_2(\Omega)} \\
		+ \norm{f_{,x_3}(t_0)}_{L_2(\Omega)} + \norm{g_{,x_3}(t_0)}_{L_2(\Omega)}\Big).
	\end{multline*}
	Thus
	\begin{multline*}
		\int_{t_0}^{t_1}\norm{v(s)}^2_{L_{\infty}(\Omega)}\, \ud s \leq \big(\norm{v(t_0)}_{H^1(\Omega)} + \norm{\omega(t_0)}_{H^1(\Omega)} + \norm{f(t_0)}_{L_2(\Omega)} + \norm{g(t_0)}_{L_2(\Omega)} \\
		+ \norm{f_{,x_3}(t_0)}_{L_2(\Omega)} + \norm{g_{,x_3}(t_0)}_{L_2(\Omega)} + 1\big)^6
	\end{multline*}
	and therefore for $t_1$ large enough we have
	\begin{equation*}
		-\frac{c_1}{2(\alpha + \beta) + \delta}(t_1 - t_0) + c_2\int_{t_0}^{t_1}\norm{v(s)}^2_{L_{\infty}(\Omega)}\,\ud s < 0,
	\end{equation*}
	which implies that
	\begin{equation*}
		\norm{v(t_1)}_{H^1(\Omega)} + \norm{\omega(t_1)}_{H^1(\Omega)} \leq c_{\nu,\alpha,\beta,I,P,\Omega}(t_1)\left(\norm{v(t_0)}_{H^1(\Omega)} + \norm{\omega(t_0)}_{H^1(\Omega)}\right),
	\end{equation*}
	where the function $c_{\nu,\alpha,\beta,I,P,\Omega}(t_1)$ has the property that $\lim_{t_1 \to \infty} c_{\nu,\alpha,\beta,I,P,\Omega}(t_1) = 0$. 
	This concludes the proof.	
\end{proof}

\begin{lem}\label{lem29}
	Suppose that $\Rot h(t_0), h(t_0), \omega(t_0) \in H^1(\Omega)$, $f_3\vert_{S_2} = 0$, $g'\vert_{S_2} = 0$. Assume that
	\begin{align*}
		&\norm{f(t)}_{L_2(\Omega)} \leq \norm{f(t_0)}_{L_2(\Omega)}e^{-(t - t_0)}, & & &\norm{f_{,x_3}(t)}_{L_2(\Omega)} &\leq \norm{f_{,x_3}(t_0)}_{L_2(\Omega)}e^{-(t - t_0)},\\
		&\norm{g(t)}_{L_2(\Omega)} \leq \norm{g(t_0)}_{L_2(\Omega)}e^{-(t - t_0)}, & & &\norm{g_{,x_3}(t)}_{L_2(\Omega)} &\leq \norm{g_{,x_3}(t_0)}_{L_2(\Omega)}e^{-(t - t_0)}
	\end{align*}
	for any $t_0 \leq t \leq t_1$. Then
	\begin{multline*}
		\norm{\Rot h(t_1)}^2_{L_2(\Omega)} + \norm{h(t_1)}^2_{L_2(\Omega)} + \norm{\theta(t_1)}^2_{L_2(\Omega)} \\
		\leq \norm{\Rot h(t_0)}^2_{L_2(\Omega)} + \norm{h(t_0)}^2_{L_2(\Omega)} + \norm{\theta(t_0)}^2_{L_2(\Omega)}.
	\end{multline*}
\end{lem}

Note, that we have assumed that the $f_3$ and $g'$ vanish on $S_2$. Both functions appear in boundary integrals which need to be estimated. Although it is possible, but it leads to the presence of $c_I\left(\norm{h}^2_{L_2(\Omega)} + \norm{\theta}^2_{L_2(\Omega)}\right)$ on the right-hand side. Since $c_I$ is out of control we are unable to ensure the uniform estimates analogously as we did in Lemma \ref{lem28}. 

\begin{proof}[Proof of Lemma \ref{lem29}]
	We multiply \eqref{p5}$_1$ by $- \triangle h$ and integrate over $\Omega$, which yields
	\begin{multline}\label{eq9}
		- \int_{\Omega} h_{,t}\cdot \triangle h\, \ud x + (\nu + \nu_r)\int_{\Omega} \abs{\triangle h}^2\, \ud x - \int_{\Omega} \nabla q \cdot \triangle h\, \ud x \\
		= \int_{\Omega} v \cdot \nabla h \cdot \triangle h\, \ud x + \int_{\Omega} h \cdot \nabla v \cdot \triangle h\, \ud x - 2\nu_r \int_{\Omega} \Rot \theta \cdot \triangle h\, \ud x - \int_{\Omega} f,_{x_3}\cdot \triangle h\, \ud x.
	\end{multline}
	For the first term on the left-hand side we have
	\begin{multline*}
		- \int_{\Omega} h_{,t}\cdot \triangle h\, \ud x = \int_{\Omega} h_{,t} \cdot \Rot \Rot h\, \ud x \\
		= \frac{1}{2}\Dt \int_{\Omega}\abs{\Rot h}^2\, \ud x + \int_{S_1} h_{,t}\cdot \Rot h \times n\, \ud S_1 + \int_{S_2} h_{2,t} \left(\Rot h\right)_1 - h_{1,t} \left(\Rot h\right)_2\, \ud S_2 \\
		= \frac{1}{2}\Dt \int_{\Omega}\abs{\Rot h}^2\, \ud x
	\end{multline*}
	where the boundary integrals vanish due to the boundary conditions \eqref{p5}$_{3,4}$.

	The third term on the left-hand side in \eqref{eq9} is equal to
	\begin{multline*}
		\int_{\Omega} \nabla q\cdot \Rot\Rot h\, \ud x = \int_{\Omega} \Rot \nabla q \cdot \Rot h\, \ud x + \int_{S_1} \nabla q\cdot \Rot h \times n\, \ud S_1 \\
		+ \int_{S_2} q_{,x_2}\left(\Rot h\right)_1 - q_{,x_1}\left(\Rot h\right)_2\, \ud S_2 = 0,
	\end{multline*}
	which follows from \eqref{p5}$_3$ and 
	\begin{equation*}
		q\vert_{S_2} = -v\cdot \nabla v \cdot n\vert_{S_2} + f_3\vert_{S_2} + 2\nu_r\Rot \omega\cdot n\vert_{S_2} = f_3\vert_{S_2} = 0
	\end{equation*}
	because we assumed that $f_3\vert_{S_2} = 0$.

	Consider next the first term on the right-hand side in \eqref{eq9}. Since $\Div v = 0$ we may integrate by parts, which yields
	\begin{multline*}
		\int_{\Omega} v \cdot \nabla h \cdot \triangle h\, \ud x \leq \norm{\nabla v}_{L_6(\Omega)}\norm{\nabla h}_{L_3(\Omega)}\norm{\Rot h}_{L_2(\Omega)} + \int_S \left(v\cdot \nabla\right) h\cdot (\Rot h \times n)\, \ud S \\
		\leq \epsilon_1 c_I \norm{\nabla h}^2_{H^1(\Omega)} + \frac{1}{4\epsilon_1}\norm{\nabla v}^2_{L_6(\Omega)}\norm{\Rot h}^2_{L_2(\Omega)} \leq \epsilon_1 c_I \norm{h}^2_{H^2(\Omega)} \\
		+ \frac{1}{4\epsilon_1}\norm{\nabla v}^2_{L_6(\Omega)}\norm{\Rot h}^2_{L_2(\Omega)} \\
		\leq \epsilon_1 c_{\Omega,I} \norm{\triangle h}^2_{L_2(\Omega)} + \frac{1}{4\epsilon_1}\norm{\nabla v}^2_{L_6(\Omega)}\norm{\Rot h}^2_{L_2(\Omega)},
	\end{multline*}
	where we used that $\norm{\nabla h}_{L_3(\Omega)} \leq \norm{\nabla h}^{\frac{1}{2}}_{L_2(\Omega)}\norm{\nabla h}^{\frac{1}{2}}_{L_6(\Omega)} \leq c_I \norm{\nabla h}_{H^1(\Omega)}$. The last inequality above is justified in light of Lemma \ref{lem110}.

	For the second term on the right-hand side in \eqref{eq9} we simply have
	\begin{equation*}
		\norm{h}_{L_3(\Omega)}\norm{\nabla v}_{L_6(\Omega)}\norm{\triangle h}_{L_2(\Omega)} \leq \epsilon_2 \norm{\triangle h}_{L_2(\Omega)}^2 + \frac{1}{4\epsilon_2} \norm{h}_{L_2(\Omega)}\norm{h}_{L_6(\Omega)}\norm{\nabla v}_{L_6(\Omega)}^2.
	\end{equation*}

	The third term is estimated as follows
	\begin{equation*}
		2\nu_r\int_{\Omega}\Rot \theta\cdot \triangle h\, \ud x \leq 2\nu_r\norm{\Rot \theta}_{L_2(\Omega)}\norm{\triangle h}_{L_2(\Omega)} \\
		\leq 2\nu_r\epsilon_3 \norm{\triangle h}^2_{L_2(\Omega)} + \frac{\nu_r}{2\epsilon_3} \norm{\Rot \theta}^2_{L_2(\Omega)}.
	\end{equation*}

	Finally, for the fourth term we have
	\begin{equation*}
		\int_{\Omega}f,_{x_3}\cdot \triangle h\, \ud x \leq \norm{f,_{x_3}}_{L_2(\Omega)}\norm{\triangle h}_{L_2(\Omega)} \leq \epsilon_4 \norm{\triangle h}^2_{L_2(\Omega)} + \frac{1}{4\epsilon_4} \norm{f,_{x_3}}^2_{L_2(\Omega)}.
	\end{equation*}
	
	Setting $\epsilon_1 c_{\Omega,I} = \epsilon_2 = \epsilon_4 = \frac{\nu}{6}$ and $\epsilon_3 = \frac{1}{2}$ yields
	\begin{multline}\label{eq31}
		\frac{1}{2}\Dt \int_{\Omega}\abs{\Rot h}^2\, \ud x + \frac{\nu}{2}\norm{\triangle h}^2_{L_2(\Omega)} \leq \frac{3c_{\Omega,I}}{2\nu}\norm{\nabla v}^2_{L_6(\Omega)}\norm{\Rot h}^2_{L_2(\Omega)} \\
		+ \frac{3}{2\nu} \norm{h}_{L_2(\Omega)}\norm{h}_{L_6(\Omega)}\norm{\nabla v}_{L_6(\Omega)}^2 + \nu_r \norm{\Rot \theta}^2_{L_2(\Omega)} + \frac{3}{2\nu}\norm{f_{,x_3}}^2_{L_2(\Omega)} \\
		\leq \frac{3c_{I,\Omega}}{\nu}\norm{\nabla v}^2_{L_6(\Omega)}\norm{\Rot h}^2_{L_2(\Omega)} + \nu_r \norm{\Rot \theta}^2_{L_2(\Omega)} + \frac{3}{2\nu}\norm{f_{,x_3}}^2_{L_2(\Omega)},
	\end{multline}
	where in the last inequality we used $\norm{h}_{L_2(\Omega)}\norm{h}_{L_6(\Omega)} \leq c_I\norm{h}_{H^1(\Omega)}^2$ and subsequently utilized Lemma \ref{l2}.
	Since $\norm{\Rot \theta}^2_{L_2(\Omega)} \leq 2 \norm{\theta}^2_{H^1(\Omega)}$ we get from \eqref{eq31}
	\begin{multline}\label{eq116}
		\frac{1}{2}\Dt \int_{\Omega}\abs{\Rot h}^2\, \ud x + \frac{\nu}{2}\norm{\triangle h}^2_{L_2(\Omega)}  \\
		\leq \frac{3c_{\Omega,I}}{\nu}\norm{\nabla v}^2_{L_6(\Omega)}\norm{\Rot h}^2_{L_2(\Omega)} + 2\nu_r \norm{\theta}^2_{H^1(\Omega)} + \frac{3}{2\nu}\norm{f_{,x_3}}^2_{L_2(\Omega)}.
	\end{multline}
	In \cite[Proof of Lemma 8.4]{2012arXiv1205.4046N} we established
	\begin{multline*}
		\frac{1}{2}\Dt \left(\int_{\Omega}h^2\, \ud x + \int_{\Omega}\theta^2\, \ud x\right) + \frac{\nu}{2c_{\Omega}}\norm{h}^2_{H^1(\Omega)} + \frac{\alpha}{2c_{\Omega}} \norm{\theta}^2_{H^1(\Omega)} \\
		\leq \frac{c_{I,\Omega}}{\nu}\norm{\nabla v}^2_{L_2(\Omega)}\norm{h}^2_{L_3(\Omega)} + \frac{c_{I,\Omega}}{\alpha}\norm{\nabla \omega}_{L_2(\Omega)}^2\norm{h}_{L_3(\Omega)}^2 \\
		+ \frac{c_{I,\Omega}}{\nu}\norm{f_{,x_3}}^2_{L_{\frac{6}{5}}(\Omega)} + \frac{c_{I,\Omega}}{\alpha}\norm{g_{,x_3}}^2_{L_{\frac{6}{5}}(\Omega)}.
	\end{multline*}
	Multiplying it by $\frac{8\nu_rc_{\Omega}}{\alpha}$ and adding to \eqref{eq116} yields
	\begin{multline*}
		\frac{1}{2}\Dt \left(\int_{\Omega}\abs{\Rot h}^2 + \frac{8\nu_rc_{\Omega}}{\alpha}\abs{h}^2 + \frac{8\nu_rc_{\Omega}}{\alpha}{\theta}^2\, \ud x\right) + \frac{\nu}{2}\norm{\triangle h}^2_{L_2(\Omega)} + \frac{4\nu\nu_rc_{\Omega}}{\alpha}\norm{h}^2_{H^1(\Omega)} + 4\nu_r\norm{\theta}^2_{H^1(\Omega)} \\
		\leq \frac{8\nu_rc_{I,\Omega}}{\nu\alpha}\norm{\nabla v}^2_{L_2(\Omega)}\norm{h}^2_{L_3(\Omega)} + \frac{8\nu_rc_{I,\Omega}}{\alpha^2}\norm{\nabla \omega}_{L_2(\Omega)}^2\norm{h}_{L_3(\Omega)}^2 \\
		+ \frac{8\nu_rc_{I,\Omega}}{\nu\alpha}\norm{f_{,x_3}}^2_{L_{\frac{6}{5}}(\Omega)} + \frac{8\nu_rc_{I,\Omega}}{\alpha^2}\norm{g_{,x_3}}^2_{L_{\frac{6}{5}}(\Omega)} \\
		+ \frac{3c_{\Omega,I}}{\nu}\norm{\nabla v}^2_{L_6(\Omega)}\norm{\Rot h}^2_{L_2(\Omega)} + 2\nu_r \norm{\theta}^2_{H^1(\Omega)} + \frac{3}{2\nu}\norm{f_{,x_3}}^2_{L_2(\Omega)}.
	\end{multline*}
	Interpolation between $L_2$ and $L_6$ shows that $\norm{h}_{L_3(\Omega)} \leq \norm{h}_{L_2(\Omega)}^{\frac{1}{2}}\norm{h}_{L_6(\Omega)}^{\frac{1}{2}} \leq c_I \norm{h}_{H^1(\Omega)}$, because
	\begin{equation*}
		\frac{1}{3} = \frac{\kappa}{2} + \frac{1 - \kappa}{6} \quad \Leftrightarrow \quad \kappa = \frac{1}{2}.
	\end{equation*}
	In view of \eqref{eq120} we see that $\norm{h}_{H^1(\Omega)} \leq c_I \norm{\Rot h}_{L_2(\Omega)}$. Since
	\begin{align*}
		\norm{f_{,x_3}}^2_{L_{\frac{6}{5}}(\Omega)} &\leq c_{\Omega} \norm{f_{,x_3}}^2_{L_2(\Omega)}, \\
		\norm{g_{,x_3}}^2_{L_{\frac{6}{5}}(\Omega)} &\leq c_{\Omega} \norm{g_{,x_3}}^2_{L_2(\Omega)} 
	\end{align*}
	we get
	\begin{multline}\label{eq32}
		\frac{1}{2}\Dt \left(\int_{\Omega}\abs{\Rot h}^2 + \frac{8\nu_rc_{\Omega}}{\alpha}\abs{h}^2 + \frac{8\nu_rc_{\Omega}}{\alpha}{\theta}^2\, \ud x\right) + \frac{\nu}{2}\norm{\triangle h}^2_{L_2(\Omega)} + \frac{4\nu\nu_rc_{\Omega}}{\alpha}\norm{h}^2_{H^1(\Omega)} \\
		+ 2\nu_r\norm{\theta}^2_{H^1(\Omega)} \\	
		\leq \left(\frac{8\nu_rc_{I,\Omega}}{\nu\alpha}\norm{\nabla v}^2_{L_2(\Omega)} + \frac{8\nu_rc_{I,\Omega}}{\alpha^2}\norm{\nabla \omega}_{L_2(\Omega)}^2 + \frac{3c_{\Omega,I}}{\nu}\norm{\nabla v}^2_{L_6(\Omega)}\right)\norm{\Rot h}_{L_2(\Omega)}^2 \\
		+ \frac{16\nu_rc_{I,\Omega} + 3\alpha}{2\nu\alpha}\norm{f_{,x_3}}^2_{L_2(\Omega)} + \frac{8\nu_rc_{I,\Omega}}{\alpha^2}\norm{g_{,x_3}}^2_{L_2(\Omega)}.
	\end{multline}
	From \eqref{eq120} and natural embeddings it follows
	\begin{equation*}
		\frac{\nu}{2}\norm{\triangle h}^2_{L_2(\Omega)} \geq \frac{\nu}{2c_{\Omega}}\norm{h}^2_{H^2(\Omega)} \geq \frac{\nu}{2c_{\Omega}}\norm{h}^2_{H^1(\Omega)} \geq \frac{\nu}{4c_{\Omega}}\norm{\Rot h}^2_{L_2(\Omega)}.
	\end{equation*}
	In the same manner
	\begin{align*}
		\frac{4\nu\nu_rc_{\Omega}}{\alpha}\norm{h}^2_{H^1(\Omega)} &\geq \frac{4\nu\nu_rc_{\Omega}}{\alpha}\norm{h}^2_{L_2(\Omega)}, \\
		2\nu_r\norm{\theta}^2_{H^1(\Omega)} &\geq 2\nu_r\norm{\theta}^2_{L_2(\Omega)}.
	\end{align*}
	Thus
	\begin{multline*}
		\frac{\nu}{4c_{\Omega}}\norm{\Rot h}^2_{L_2(\Omega)} + \frac{4\nu\nu_rc_{\Omega}}{\alpha}\norm{h}^2_{L_2(\Omega)} + 2\nu_r\norm{\theta}^2_{L_2(\Omega)} \\
		\geq \frac{\nu}{4c_{\Omega}} \left(\norm{\Rot h}^2_{L_2(\Omega)} + \frac{16\nu_rc_{\Omega}^2}{\alpha}\norm{h}^2_{L_2(\Omega)} + \frac{8\nu_rc_{\Omega}}{\nu}\norm{\theta}^2_{L_2(\Omega)}\right) \\
		= \frac{\nu}{4c_{\Omega}} \left(\norm{\Rot h}^2_{L_2(\Omega)} + 2c_{\Omega}\frac{8\nu_rc_{\Omega}}{\alpha}\norm{h}^2_{L_2(\Omega)} + \frac{\alpha}{\nu}\frac{8\nu_rc_{\Omega}}{\alpha}\norm{\theta}^2_{L_2(\Omega)}\right) \\
		\geq \frac{\min\left\{\nu,2\nu c_{\Omega},\alpha\right\}}{4c_{\Omega}} \left(\norm{\Rot h}^2_{L_2(\Omega)} + \frac{8\nu_rc_{\Omega}}{\alpha}\norm{h}^2_{L_2(\Omega)} + \frac{8\nu_rc_{\Omega}}{\alpha}\norm{\theta}^2_{L_2(\Omega)}\right).
	\end{multline*}
	Denote
	\begin{equation*}
		X(t) = \norm{\Rot h(t)}^2_{L_2(\Omega)} + \frac{8\nu_rc_{\Omega}}{\alpha}\norm{h(t)}^2_{L_2(\Omega)} + \frac{8\nu_rc_{\Omega}}{\alpha}\norm{\theta(t)}^2_{L_2(\Omega)}.
	\end{equation*}
	Then \eqref{eq32} becomes
	\begin{multline*}
		\frac{1}{2}\Dt X(t) + \frac{\min\left\{\nu,2\nu c_{\Omega},\alpha\right\}}{4c_{\Omega}} X(t) \\
		\leq \left(\frac{8\nu_rc_{I,\Omega}}{\nu\alpha}\norm{\nabla v}^2_{L_2(\Omega)} + \frac{8\nu_rc_{I,\Omega}}{\alpha^2}\norm{\nabla \omega}_{L_2(\Omega)}^2 + \frac{3c_{\Omega,I}}{\nu}\norm{\nabla v}^2_{L_6(\Omega)}\right)X(t) \\
		+ \frac{16\nu_rc_{I,\Omega} + 3\alpha}{2\nu\alpha}\norm{f_{,x_3}}^2_{L_2(\Omega)} + \frac{8\nu_rc_{I,\Omega}}{\alpha^2}\norm{g_{,x_3}}^2_{L_2(\Omega)}.
	\end{multline*}
	Let
	\begin{align*}
		c_1 &= \frac{\min\left\{\nu,2\nu c_{\Omega},\alpha\right\}}{2c_{\Omega}}, \\
		c_2 &= \max\left\{\frac{16\nu_rc_{I,\Omega}}{\nu\alpha},\frac{16\nu_rc_{I,\Omega}}{\alpha^2},\frac{6c_{\Omega,I}}{\nu}\right\}, \\
		c_3 &= \max\left\{\frac{16\nu_rc_{I,\Omega} + 3\alpha}{\nu\alpha},\frac{16\nu_rc_{I,\Omega}}{\alpha^2}\right\}.
	\end{align*}
	Then, the last inequality implies that
	\begin{multline*}
		\Dt \left(X(t)e^{c_1t - c_2\left(\int_{t_0}^t \norm{\nabla v(s)}^2_{L_2(\Omega)} + \norm{\nabla \omega(s)}^2_{L_2(\Omega)} + \norm{\nabla v(s)}^2_{L_6(\Omega)}\, \ud s\right)}\right) \\
		\leq c_3 \left(\norm{f_{,x_3}}^2_{L_2(\Omega)} + \norm{g_{,x_3}}^2_{L_2(\Omega)}\right)e^{c_1t - c_2\left(\int_{t_0}^t \norm{\nabla v(s)}^2_{L_2(\Omega)} + \norm{\nabla \omega(s)}^2_{L_2(\Omega)} + \norm{\nabla v(s)}^2_{L_6(\Omega)}\, \ud s\right)}.
	\end{multline*}
	Integrating with respect to $t \in (t_0,t_1)$ yields
	\begin{multline*}
		X(t_1)e^{c_1t_1 - c_2\left(\int_{t_0}^{t_1} \norm{\nabla v(s)}^2_{L_2(\Omega)} + \norm{\nabla \omega(s)}^2_{L_2(\Omega)} + \norm{\nabla v(s)}^2_{L_6(\Omega)}\, \ud s\right)} \\
		\leq c_3 \int_{t_0}^{t_1}\left(\norm{f_{,x_3}(t)}^2_{L_2(\Omega)} + \norm{g_{,x_3}(t)}^2_{L_2(\Omega)}\right)e^{c_1t - c_2\left(\int_{t_0}^t \norm{\nabla v(s)}^2_{L_2(\Omega)} + \norm{\nabla \omega(s)}^2_{L_2(\Omega)} + \norm{\nabla v(s)}^2_{L_6(\Omega)}\, \ud s\right)}\, \ud t \\
		+ X(t_0)e^{c_1t_0}.
	\end{multline*}
	Since
	\begin{equation*}
		\int_{t_0}^t \norm{\nabla v(s)}^2_{L_2(\Omega)} + \norm{\nabla \omega(s)}^2_{L_2(\Omega)} + \norm{\nabla v(s)}^2_{L_6(\Omega)}\, \ud s \geq 0
	\end{equation*}
	we get that
	\begin{multline*}
		X(t_1) \leq c_3e^{ - c_1t_1 + c_2\left(\int_{t_0}^{t_1} \norm{\nabla v(s)}^2_{L_2(\Omega)} + \norm{\nabla \omega(s)}^2_{L_2(\Omega)} + \norm{\nabla v(s)}^2_{L_6(\Omega)}\, \ud s\right)}  \\
		\cdot \int_{t_0}^{t_1}\left(\norm{f_{,x_3}(t)}^2_{L_2(\Omega)} + \norm{g_{,x_3}(t)}^2_{L_2(\Omega)}\right)e^{c_1t}\, \ud t \\
		+ X(t_0) e^{ - c_1(t_1 - t_0) + c_2\left(\int_{t_0}^{t_1} \norm{\nabla v(s)}^2_{L_2(\Omega)} + \norm{\nabla \omega(s)}^2_{L_2(\Omega)} + \norm{\nabla v(s)}^2_{L_6(\Omega)}\, \ud s\right)}.
	\end{multline*}
	By assumption on $f_{,x_3}$ and $g_{,x_3}$ we obtain
	\begin{multline}\label{eq33}
		X(t_1) \leq c_3e^{ - c_1t_1 + c_2\left(\int_{t_0}^{t_1} \norm{\nabla v(s)}^2_{L_2(\Omega)} + \norm{\nabla \omega(s)}^2_{L_2(\Omega)} + \norm{\nabla v(s)}^2_{L_6(\Omega)}\, \ud s\right)}  \\
		\cdot \int_{t_0}^{t_1}e^{c_1t}e^{-(t - t_0)}\, \ud t \left(\norm{f_{,x_3}(t_0)}^2_{L_2(\Omega)} + \norm{g_{,x_3}(t_0)}^2_{L_2(\Omega)}\right)\\
		+ X(t_0) e^{ - c_1(t_1 - t_0) + c_2\left(\int_{t_0}^{t_1} \norm{\nabla v(s)}^2_{L_2(\Omega)} + \norm{\nabla \omega(s)}^2_{L_2(\Omega)} + \norm{\nabla v(s)}^2_{L_6(\Omega)}\, \ud s\right)}.
	\end{multline}
	Next we estimate the integral with respect to $t$
	\begin{multline*}
		e^{-c_1t_1} \int_{t_0}^{t_1} e^{c_1t}e^{-(t - t_0)}\, \ud t = e^{-c_1t_1 + t_0}\int_{t_0}^{t_1} e^{t(c_1 - 1)}\, \ud t = e^{-c_1t_1 + t_0}\frac{1}{c_1 - 1} e^{t(c_1 - 1)}\bigg\vert_{t_0}^{t_1} \\
		= e^{-c_1t_1 + t_0}\frac{1}{c_1 - 1} \left(e^{t_1(c_1 - 1)} - e^{t_0(c_1 - 1)}\right) \\
		\leq \begin{cases}
					e^{-c_1t_1 + t_0}\frac{1}{c_1 - 1} e^{c_1t_1 - t_1} = \frac{1}{c_1 - 1}e^{-(t_1 - t_0)} & \text{for } c_1 > 1, \\
					e^{-c_1t_1 + t_0}\frac{1}{1 - c_1} e^{c_1t_0 - t_0} = \frac{1}{1 - c_1}e^{-c_1(t_1 - t_0)} & \text{for } c_1 < 1.
				\end{cases}
	\end{multline*}
	Consider the quantity
	\begin{equation}\label{eq130}
		\int_{t_0}^{t_1} \norm{\nabla v(s)}^2_{L_2(\Omega)} + \norm{\nabla \omega(s)}^2_{L_2(\Omega)} + \norm{\nabla v(s)}^2_{L_6(\Omega)}\, \ud s.
	\end{equation}
	In view of Lemma \ref{l4} we see
	\begin{multline*}
		\int_{t_0}^{t_1} \norm{\nabla v(s)}^2_{L_2(\Omega)} + \norm{\nabla \omega(s)}^2_{L_2(\Omega)}\, \ud s \leq c_{\nu,\alpha,I,\Omega}E_{v,\omega}^2(t_1) \\
		\leq \int_{t_0}^{t_1}\norm{f(s)}^2_{L_2(\Omega)} + \norm{g(s)}^2_{L_2(\Omega)}\, \ud s + \norm{v(t_0)}^2_{L_2(\Omega)} + \norm{\omega(t_0)}^2_{L_2(\Omega)} \\
		\leq \int_{t_0}^{t_1}e^{-(s - t_0)}\, \ud s \left(\norm{f(t_0)}^2_{L_2(\Omega)} + \norm{g(t_0)}^2_{L_2(\Omega)}\right)+ \norm{v(t_0)}^2_{L_2(\Omega)} + \norm{\omega(t_0)}^2_{L_2(\Omega)} \\
		\leq \norm{f(t_0)}^2_{L_2(\Omega)} + \norm{g(t_0)}^2_{L_2(\Omega)} + \norm{v(t_0)}^2_{L_2(\Omega)} + \norm{\omega(t_0)}^2_{L_2(\Omega)}.
	\end{multline*}
	To estimate the last term in \eqref{eq130} we use the estimate from Theorem \ref{t1}
\begin{multline*}
		\int_{t_0}^{t_1}\norm{\nabla v(s)}^2_{L_6(\Omega)}\, \ud s \leq c_I\norm{v}^2_{L_2(t_0,t_1;H^2(\Omega))} \leq c_I\norm{v}^2_{W^{2,1}_2(\Omega^{t_1})} \\
		\leq c_{\alpha,\nu,\nu_r,I,P,\Omega}\Big(E_{v,\omega}(t_1) + E_{h,\theta}(t_1) + \norm{f'}_{L_2(\Omega^{t_1})} + \norm{v(t_0)}_{H^1(\Omega)} + 1 \Big)^6.
	\end{multline*}
	By assumption on $f$, $g$, their derivative with respect to $x_3$ we see that 
	\begin{multline*}
		E_{v,\omega}(t_1) + E_{h,\theta}(t_1)\leq c_{\Omega}\Big(\norm{v(t_0)}_{H^1(\Omega)} + \norm{\omega(t_0)}_{H^1(\Omega)} + \norm{f(t_0)}_{L_2(\Omega)} + \norm{g(t_0)}_{L_2(\Omega)} \\
		+ \norm{f_{,x_3}(t_0)}_{L_2(\Omega)} + \norm{g_{,x_3}(t_0)}_{L_2(\Omega)}\Big).
	\end{multline*}
	Thus,
	\begin{multline*}
		\int_{t_0}^{t_1}\norm{\nabla v(s)}^2_{L_6(\Omega)}\, \ud s \leq \Big(\norm{v(t_0)}_{H^1(\Omega)} + \norm{\omega(t_0)}_{H^1(\Omega)} + \norm{f(t_0)}_{L_2(\Omega)} + \norm{g(t_0)}_{L_2(\Omega)} \\
		+ \norm{f_{,x_3}(t_0)}_{L_2(\Omega)} + \norm{g_{,x_3}(t_0)}_{L_2(\Omega)} + 1\Big)^6
	\end{multline*}
	and therefore for $t_1$ large enough we have
	\begin{equation*}
		-\min\{1,c_1\}(t_1 - t_0) + c_2\int_{t_0}^{t_1} \norm{\nabla v(s)}^2_{L_2(\Omega)} + \norm{\nabla \omega(s)}^2_{L_2(\Omega)} + \norm{\nabla v(s)}^2_{L_6(\Omega)}\, \ud s < 0,
	\end{equation*}
	which combined with \eqref{eq33} implies that
	\begin{equation*}
		X(t_1) \leq c_{\alpha,\nu,\nu_r,I,P,\Omega}(t_1) X(t_0),
	\end{equation*}
	where the function $c_{\alpha,\nu,\nu_r,I,P,\Omega}(t_1)$ has the property
	\begin{equation*}
		\lim_{t_1 \to \infty} c_{\alpha,\nu,\nu_r,I,P,\Omega}(t_1) = 0.
	\end{equation*}
	From the above inequality we immediately get that
	\begin{multline*}
		\min\left\{1,\frac{8\nu_rc_{\Omega}}{\alpha}\right\} \left(\norm{\Rot h(t_1)}^2_{L_2(\Omega)} + \norm{h(t_1)}^2_{L_2(\Omega)} + \norm{\theta(t_1)}^2_{L_2(\Omega)}\right) \\
		\leq c_{\alpha,\nu,\nu_r,I,P,\Omega}\max\left\{1,\frac{8\nu_rc_{\Omega}}{\alpha}\right\} \left(\norm{\Rot h(t_0)}^2_{L_2(\Omega)} + \norm{h(t_0)}^2_{L_2(\Omega)} + \norm{\theta(t_0)}^2_{L_2(\Omega)}\right).
	\end{multline*}
	For $t_1 \geq t^*$ the inequality yields
	\begin{multline*}
		\norm{\Rot h(t_1)}^2_{L_2(\Omega)} + \norm{h(t_1)}^2_{L_2(\Omega)} + \norm{\theta(t_1)}^2_{L_2(\Omega)} \\
		\leq \norm{\Rot h(t_0)}^2_{L_2(\Omega)} + \norm{h(t_0)}^2_{L_2(\Omega)} + \norm{\theta(t_0)}^2_{L_2(\Omega)},
	\end{multline*}
	which is our claim.
\end{proof}

\begin{proof}[Proof of Theorem \ref{t2}]
	From Lemma \ref{lem29} it follows that
	\begin{multline*}
		\sup_k \left(\norm{\Rot h(kT)}^2_{L_2(\Omega)} + \norm{h(kT)}^2_{L_2(\Omega)} + \norm{\theta(kT)}^2_{L_2(\Omega)}\right) \\
		\leq \norm{\Rot h(0)}^2_{L_2(\Omega)} + \norm{h(0)}^2_{L_2(\Omega)} + \norm{\theta(0)}^2_{L_2(\Omega)}.
	\end{multline*}By assumption 
	\begin{align*}
		&\sup_k \norm{f_{,x_3}}_{L_2(\Omega\times(kT,(k + 1)T))} < \infty, \\
		&\sup_k \norm{g_{,x_3}}_{L_2(\Omega\times(kT,(k + 1)T))} < \infty.
	\end{align*}
	Thus, we set
	\begin{equation*}
		\sup_k \delta(kT) \leq \delta(0) =: \delta.
	\end{equation*}
	Let $t_0 = 0$, $t_1 = T$. In view of Theorem \ref{t1} with $\delta$ we get the existence of regular solution on the interval $[0,T]$. Lemma \ref{lem22} yields the inequality
	\begin{equation*}
		\norm{v(T)}^2_{H^1(\Omega)} + \norm{\omega(T)}^2_{H^1(\Omega)} < \norm{v(0)}^2_{H^1(\Omega)} + \norm{\omega(0)}^2_{H^1(\Omega)}
	\end{equation*}
	for $T > 0$ sufficiently large. It allows us to use Theorem \ref{t1} with the initial conditions $v(T)$, $\omega(T)$ and with $\delta$ on the time interval $[T,2T]$. From Lemma \ref{lem22} it follows that
	\begin{equation*}
		\norm{v(2T)}^2_{H^1(\Omega)} + \norm{\omega(2T)}^2_{H^1(\Omega)} < \norm{v(T)}^2_{H^1(\Omega)} + \norm{\omega(T)}^2_{H^1(\Omega)},
	\end{equation*}
	which in view of previous estimate provides us with regular solution on the interval $[0,2T]$. Reiterating this procedure yields the existence of regular solution on $[0,kT]$. Passing with $k$ up to infinity provides the global existence. 
\end{proof}


\begin{thebibliography}{OTRM97}

\bibitem[ATS73]{ari}
T.~Ariman, M.A. Turk, and N.D. Sylvester, \emph{Microcontinuum fluid mechanics
  --- a review.}, Int. J. Eng. Sci. \textbf{11} (1973), 905--930.

\bibitem[BS70]{bou}
G.C. Boulougouris and J.~Sevilla, \emph{{Velocity distribution and other
  characteristics of steady and pulsatille blood in fine glas tubes.}},
  Biorheology \textbf{7} (1970), no.~2, 85--107.

\bibitem[CD00]{chol}
J.W. Cholewa and T.~D{\l}otko, \emph{Global attractors in abstract parabolic
  problems.}, Cambridge: Cambridge University Press, 2000.

\bibitem[CMR98]{Clopeau:1998vj}
T.~Clopeau, A.~Mikeli{\'c}, and R.~Robert, \emph{On the vanishing viscosity
  limit for the 2d incompressible navier-stokes equations with the friction
  type boundary conditions}, Nonlinearity \textbf{11} (1998), no.~6,
  1625--1636.

\bibitem[Eri66]{erin}
A.C. Eringen, \emph{{Theory of micropolar fluids.}}, J. Math. Mech. \textbf{16}
  (1966), 1--16.

\bibitem[Kel06]{kell}
J.P. Kelliher, \emph{Navier-stokes equations with navier boundary conditions
  for a bounded domain in the plane.}, SIAM J. Math. Anal. \textbf{38} (2006),
  no.~1, 210--232.

\bibitem[Lan76]{lang1}
H.~Lange, \emph{Die {E}xistenz von {L}{\"o}sungen der {G}leichungen, welche die
  {S}tr{\"o}mung inkompressibler mikropolarer {F}l{\"u}ssigkeiten
  beschreiben.}, Z. Angew. Math. Mech. \textbf{56} (1976), no.~4, 129--139.

\bibitem[Lan77]{lang2}
\bysame, \emph{The existence of instationary flows in incompressible micropolar
  fluids.}, Arch. Mech. (Arch. Mech. Stos.) \textbf{29} (1977), no.~5,
  741--744.

\bibitem[LL87]{land}
L.D. Landau and E.M. Lifshitz, \emph{{Fluid Mechanics: Volume 6 (Course of
  Theoretical Physics).}}, {A Butterworth-Heinemann Title; 2 edition}, 1987.

\bibitem[LSU67]{lad}
O.A. Lady{\v z}enskaja, V.A. Solonnikov, and N.N. Ural'ceva, \emph{Linear and
  quasilinear equations of parabolic type}, Translated from the Russian by S.
  Smith. Translations of Mathematical Monographs, Vol. 23, American
  Mathematical Society, Providence, R.I., 1967.

\bibitem[{\L}uk89]{luk11}
G.~{\L}ukaszewicz, \emph{On the existence, uniquenss and asymptotic properties
  for solutions of flows of asymmetric fluids.}, Rend. Accad. Naz. Sci. Detta
  XL, V. Ser. \textbf{13} (1989), no.~1, 105--120.

\bibitem[{\L}uk99]{luk1}
\bysame, \emph{Micropolar fluids. theory and applications.}, Boston:
  Birkh{\"a}user, 1999.

\bibitem[Mig84]{mig}
N.P. Migun, \emph{On hydrodynamic boundary conditions for microstructural
  fluids.}, Rheol. Acta \textbf{23} (1984), 575--581.

\bibitem[{Now}12]{2012arXiv1205.4046N}
B.~{Nowakowski}, \emph{{Large time existence of strong solutions to micropolar
  equations in cylindrical domains}}, ArXiv e-prints (2012), Submitted.

\bibitem[OTRM97]{ort}
E.E. Ortega-Torres and M.A. Rojas-Medar, \emph{Magneto-micropolar fluid motion:
  Global existence of strong solutions.}, Abstr. Appl. Anal. \textbf{4} (1997),
  no.~2, 109--125.

\bibitem[PRU74]{pop2}
A.S. Popel, S.A Regirer, and P.I. Usick, \emph{{A continuum model of blood
  flow.}}, Biorheology \textbf{11} (1974), 427--437.

\bibitem[RM97]{roj}
M.A. Rojas-Medar, \emph{Magneto-micropolar fluid motion: Existence and
  uniqueness of strong solution.}, Math. Nachr. \textbf{188} (1997), 301--319.

\bibitem[RZ08]{ren1}
J.~Renc{\l}awowicz and W.M. Zaj{\k a}czkowski, \emph{Large time regular
  solutions to the navier-stokes equations in cylindrical domains.}, Topol.
  Methods Nonlinear Anal. \textbf{32} (2008), no.~1, 69--87.

\bibitem[SASM02]{shar}
K.V. Sharp, R.J. Adrian, J.G. Santiago, and J.I. Molho, \emph{{Liquid Flows in
  Microchannels.}}, {The MEMS Handbook } (Mohamed~Gad ed~Hak, ed.), CRC Press,
  2002.

\bibitem[Sav78]{sav2}
V.A. Sava, \emph{The initial-boundary-value problems in the theory of
  micropolar fluids.}, Z. Angew. Math. Mech. \textbf{58} (1978), 511--518.

\bibitem[Sol65]{sol2}
V.A. Solonnikov, \emph{On boundary value problems for linear parabolic systems
  of differential equations of general form.}, Proc. Steklov Inst. Math.
  \textbf{83} (1965), 1--184, Translation from: Trudy Mat. Inst. Steklov 83
  (1965), 3--163, (Russian).

\bibitem[Sol73]{sol1}
\bysame, \emph{Overdetermined elliptic boundary-value problems.}, Journal of
  Mathematical Sciences \textbf{1} (1973), no.~4, 477--512, Translation from:
  Zap. Nauchn. Sem. LOMI 21 (1971), 112--158, (Russian).

\bibitem[Tem79]{tem}
R.~Temam, \emph{Navier-{S}tokes equations.}, revised ed., Studies in
  Mathematics and its Applications, vol.~2, North-Holland Publishing Co.,
  Amsterdam, 1979, Theory and numerical analysis, With an appendix by F.
  Thomasset.

\bibitem[Yam05]{yam}
N.~Yamaguchi, \emph{Existence of global strong solution to the micropolar fluid
  system in a bounded domain.}, Math. Methods Appl. Sci. \textbf{28} (2005),
  no.~13, 1507--1526.

\bibitem[Zaj05]{wm2}
W.M. Zaj{\k a}czkowski, \emph{Long time existence of regular solutions to
  navier-stokes equations in cylindrical domains under boundary slip
  conditions.}, Stud. Math. \textbf{169} (2005), no.~3, 243--285.

\bibitem[Zaj11]{wm6}
\bysame, \emph{{On global regular solutions to the Navier-Stokes equations in
  cylindrical domains.}}, Topol. Methods Nonlinear Anal. \textbf{37} (2011),
  no.~1, 55--86.

\end{thebibliography}
\end{document}